\documentclass[11pt,leqno]{article}
\usepackage{geometry}
\newgeometry{vmargin={28mm}, hmargin={35mm,35mm}}   

\usepackage{xcolor}

\usepackage[hidelinks]{hyperref}

\usepackage{amsmath,amsthm,amsfonts,amssymb,latexsym,amscd,enumerate}
\usepackage{slashed}
\usepackage{mathtools}

\usepackage{palatino}
\usepackage[mathcal]{euler}

\usepackage{xy}
\xyoption{all}

\numberwithin{equation}{section}

\newtheorem{theorem}{Theorem}[section]
\newtheorem*{theorem*}{Theorem}
\newtheorem{corollary}[theorem]{Corollary}
\newtheorem{proposition}[theorem]{Proposition}
\newtheorem*{proposition*}{Proposition}
\newtheorem{lemma}[theorem]{Lemma}

\theoremstyle{definition}
\newtheorem{definition}[theorem]{Definition}

\newtheorem{example}[theorem]{Example}

\newtheorem{remark}[theorem]{Remark}

\DeclareMathOperator{\End}{End}

\DeclareMathOperator{\holonomy}{Hol}

\begin{document}

\title{Enlargeable Foliations and the Monodromy Groupoid}

\author{Guangxiang Su and Zelin Yi}

\date{}

\maketitle


\abstract{
	Let $M$ be a spin manifold, the Dirac operator with coefficient  in the universal flat Hilbert $C^\ast \pi_1$-module determines a Rosenberg index element which, according to B. Hanke and T. Schick, subsumes the enlargeablility obstruction of positive scalar curvature on $M$. In this note, we generalize this result to the case of spin foliation. More precisely, given a foliation $(M,F)$ with $F$ spin, we shall define a foliation version of Rosenberg index element and prove that it is nonzero at the presence of compactly enlargeability of $(M,F)$.
}

\maketitle

\section{Introduction}
If $M$ is an even dimensional spin manifold with the fundamental group $\pi_1$ and the Dirac operator 
\[
D: C^\infty_c(M,S^+) \to C^\infty_c(M,S^-),
\]
according to \cite{MiscenkoFomenko79}, the Dirac-type operator twisted by the canonical flat $C^\ast \pi_1$-bundle 
\begin{equation}\label{eq-caonical-flat-bundle}
\widetilde{M} \times_{\pi_1} C^\ast \pi_1
\end{equation}
determines an element $[\alpha(M)]$ (we shall simply write $[\alpha]$ when there is no confusion) in $K_0(C^\ast \pi_1)$ which is usually called the Rosenberg index element. In fact, if $M$ is of odd dimensional, by replacing $M$ with $M\times S^1$, $[\alpha]\in K_1(C^\ast \pi_1)$ can also be defined. The main result of \cite{HankeSchick06} is the following.

\begin{theorem}[{\cite[Prop~4.2]{HankeSchick06}}]
	If $M$ is a compactly enlargeable spin manifold, then $[\alpha] \neq 0$ in $K_n(C^\ast \pi_1)$ where $n$ is the dimension of $M$.
\end{theorem}

Let us summarize the main idea of \cite{HankeSchick06} as follows: If $M$ is compactly enlargeable, there exists a sequence of almost flat vector bundles $E_i$ of dimension $d_i$ over $M$ (almost flat means the norms of curvatures of $E_i$'s converge to zero as $i\to \infty$).   Moreover, the sequence $E_i$ can be chosen so that all Chern classes vanish except the top degree part: $c_{\operatorname{top}}(E_i) \neq 0$. Let $P_i$ be the principal frame bundle of $E_i$, $\mathcal{K}$ be the $C^\ast$-algebra of compact operators. Unitary matrices act on $\mathcal{K}$ by the inclusion $U(d_i)\hookrightarrow \mathcal{K}$. Denote by $q_i$ the image of $1\in U(d_i)$ inside $\mathcal{K}$. The associated product 
\begin{equation}\label{eq-flat-bundle-K}
	P_i\times_{U(d_i)} \mathcal{K}
\end{equation}
is a Hilbert $\mathcal{K}$-module bundle. 

\begin{definition}\label{def-algebra-A-Q}
	Let $A$ be the $C^\ast$-algebra of bounded sequence of compact operators. Namely 
	\[
	A=\left\{(a_i)\in \prod_{\mathbb{N}} \mathcal{K} : \sup_{i\in \mathbb{N}}||a_i|| < \infty\right\}.
	\]
	Let $A_i\subset A$ be the subalgebra of sequences such that all but the $i$-th component vanish. It is clear that $A_i \cong \mathcal{K}$ for all $i \in\mathbb{N}$.	Let $A^\prime \subset A$ be the subalgebra consisting of sequences that converge to zero. In other word, $A^\prime$ is the closure of 
	$
	{\bigoplus \mathcal{K}}\subset A.
	$
	Let $Q$ be the quotient $C^\ast$-algebra $A/A^\prime$.
\end{definition}

Thanks to the boundedness of the curvatures of $E_i$'s, the sequence of Hilbert module bundles $P_i\times_{U(d_i)} \mathcal{K}$ can be assembled into a Hilbert $A$-module bundle $V$. The almost flatness of $E_i$ is reflected in the fact that the curvature of $V$ is endomorphism of $A$ which take value in $\hom(A, A^\prime)$. Therefore, $V$ can be reduced into a genuinely flat Hilbert $Q$-module bundle $W=V/V\cdot A^\prime$. Thanks to the flatness of $W$, there is a holonomy representation of the fundamental group $\pi_1$ and correspondingly, a $C^\ast$-algebras homomorphism $C^\ast \pi_1 \to Q$. To detect the non-vanishing of $[\alpha]$, it is enough to show the non-vanishing of its image under the map
\begin{equation}\label{eq-k-map-homomorphism}
K_0(C^\ast \pi_1) \to K_0(Q).
\end{equation}
It is known (\cite[Prop~3.6]{HankeSchick06}) that the $K$-theory of $Q$ is explicitly computable as a quotient of $\prod \mathbb{Z}$, and the $i$-th argument of $[\alpha]$ in $K_0(Q)$ is computed as the index of the Dirac-type operator twisted by $E_i$. The non-vanishing of $[\alpha]$ then follows from the non-vanishing of the top degree Chern class and the Atiyah-Singer index theorem.

In this paper, we shall generalize the above result to the case of compactly enlargeable foliations by following the same path. If $(M,F)$ is a compactly enlargeable foliation (see Definition~\ref{def-enlargeable-foliation}) with $F$ spin. Let $G_M$ and $G_H$ be the monodromy groupoid and the holonomy groupoid of $(M,F)$ respectively. The leafwise Dirac operator
\begin{equation*}
	D: C^\infty(M,S^+(F)) \to C^\infty(M,S^-(F))
\end{equation*}
defines a $K$-theory element  $[\alpha(M,F)]$ (we shall simply write $[\alpha]$ when there is no confusion) in $K_0(C^\ast G_M)$. We shall prove in this paper that the compactly enlargeability of $(M, F)$ implies that $[\alpha]\neq 0$ in $K_0(C^\ast G_M)$.

\begin{definition}\label{def-algebras-qaq}
	Recall that $q_i\in \mathcal{K}$ is the image of $1\in U(d_i)$ inside $\mathcal{K}$. Let $q=(q_1,q_2,\cdots)\in A$, then $qAq$ is an unital $C^\ast$-algebra. We shall write $qA^\prime q= A^\prime\cap qAq$ which is an ideal in $qAq$. Finally $qQq=qAq/qA^\prime q$.
\end{definition}

The enlargeability condition gives a sequence of leafwise almost flat vector bundles $\{E_i\}$ of dimension $d_i$ whose Chern classes vanish except the top degree part. Then the sequence of principal frame bundles and their associated product with the truncated compact operators $q_i\mathcal{K}q_i$ can be defined in the same way as in \eqref{eq-flat-bundle-K}.  Again, the sequence $\{E_i\}$ can be assembled into a Hilbert $qAq$-module bundle $V$, and the almost leafwise flatness will be reflected in a genuinely leafwise flat Hilbert $qQq$-module bundle $W=V/V\cdot qA^\prime q$. However, to the best of the authors' knowledge, there is no characterization of leafwise flat vector bundle in the form of \eqref{eq-caonical-flat-bundle}. To find the counterpart of \eqref{eq-k-map-homomorphism}, we shall make use of basic $KK$-theory.

The foliation counterpart of universal cover and fundamental group is the monodromy groupoid $G_M$. The role of $C^\ast$-algebras $A,A^\prime, Q$ will be played by three crossed product $C^\ast$-algebras $C^\ast(G_M, qAq), C^\ast(G_M, qA^\prime q), C^\ast(G_M, qQq)$ which is constructed by taking the completion of the algebra of compactly supported smooth maps on $G_M$ with values in $C^\ast$-algebras $qAq,qA^\prime q, qQq$ respectively. The fact that $qAq$ is an unital $C^\ast$-algebras is crucial in the corresponding pseudodifferential calculus that we shall need.
If $W$ is a leafwise flat Hilbert $qQq$-module bundle over $M$, the space of smooth compactly supported sections of the pull back bundle $r^\ast W \to G_M$ can be completed into a Hilbert $C^\ast (G_M, qQq)$-module $\mathcal{E}_W$. It can be shown, due to the leafwise flatness, this module also has a left $C^\ast G_M$-action which, together with the zero operator, determines a $KK$-theory element in 
\begin{equation}\label{eq-KK-map-flat}
	KK(C^\ast G_M, C^\ast (G_M, qQq)),
\end{equation}
which will play the role of the map  \eqref{eq-k-map-homomorphism}. The sequence of $C^\ast$-algebras
\begin{equation*}
	0 \to C^\ast( G_M, qA^\prime q) \to C^\ast (G_M, qAq) \to C^\ast (G_M, qQq )\to 0
\end{equation*}
is exact (see Proposition~\ref{prop-exact-sequence-of-new-C-algebra}) and induces the following exact sequence at the level of $K$-theory
\begin{equation}\label{eq-exact-sequence-K-theory-tensor}
	K_0(C^\ast ( G_M, qA^\prime q)) \to K_0(C^\ast ( G_M, qA q)) \to K_0(C^\ast ( G_M, qQ q)).
\end{equation}
The image of $[\alpha]$ under the map
\begin{equation*}
	K_0(C^\ast G_M) \to K_0(C^\ast ( G_M, qQ q)),
\end{equation*}
which is given by Kasparov product with the $KK$-element \eqref{eq-KK-map-flat}, is given by the twisted leafwise Dirac operator $[D_W]\in K_0(C^\ast ( G_M, qQ q))$. The twisted leafwise Dirac operator $D_V$ defines an element in $K_0(C^\ast ( G_M, qA q))$ which is mapped to $[D_W]$ under the second map of \eqref{eq-exact-sequence-K-theory-tensor}. It suffices to show that $[D_V]$ is not in the image of the first map of \eqref{eq-exact-sequence-K-theory-tensor}.  Indeed, let $p_i$ be the projection $A \to \mathcal{K}$ into the $i$-th component, it induces a map $C^\ast ( G_M, qA q) \to C^\ast (G_M, M_{d_i}(\mathbb{C}))$. Then the image of $[D_V]$ under the composition 
$$
K_0(C^\ast ( G_M, qA q))\to \prod K_0(C^\ast (G_M, M_{d_i}(\mathbb{C}))) \cong \prod K_0(C^\ast G_M)\to \prod K_0(C^\ast_r G_H)
$$
is given by the longitudinal indices of the leafwise Dirac type operators twisted by the vector bundles $E_i$. While the image of $K_0(C^\ast(G_M, qA^\prime q))$ can be shown to be contained in the direct sum $\bigoplus K_0(C^\ast_r G_H)$. The non-vanishing of $[\alpha]$ is then a consequence of the non-vanishing of the top degree Chern classes.

This paper is organized as follows. In Section~\ref{sec-monodromy-groupoids}, we briefly recall the definition of monodromy groupoids and holonomy groupoids of a foliated manifold. In Section~\ref{sec-groupoid-algebra}, we review the notion of Haar system on Lie groupoids, the construction of full and reduced groupoid $C^\ast$-algebras and introduce $C^\ast(G,A)$ groupoid $C^\ast$-algebras with coefficient in another $C^\ast$-algebra. In Section~\ref{sec-rosenberg-index}, under the assumption that $(M,F)$ is a foliation with $F$ spin and even dimensional, we define the foliation counterpart of Rosenberg index $[\alpha] \in K_0(C^\ast G_M)$ and relate it to the longitudinal index element. In Section~\ref{sec-twisted-rosenberg-index}, we define the Rosenberg index twisted by a Hilbert $C^\ast$-module bundle. In Section~\ref{sec-hilbert-module}, we construct a Hilbert module at the presence of a leafwise flat Hilbert $Q$-module bundle. This Hilbert module will later determines a $KK$-theory element  which play the role of \eqref{eq-k-map-homomorphism}. In Section \ref{sec-enlargeable-foliation}, we write down the definition of the genuinely leafwise flat Hilbert $Q$-module bundle out of the enlargeability of $(M,F)$ and prove the non-vanishing of $[\alpha]$. In Section~\ref{sec-reduction}, we deal with odd dimensional $F$. We define $[\alpha]\in K_1(C^\ast G_M)$ and show how to reduce the non-vanishing problem to the even dimensional case.

\section{Monodromy groupoids and holonomy groupoids}\label{sec-monodromy-groupoids}
Let $(M,F)$ be a compact foliation, we shall denote the monodromy groupoid by $G_M$ and the holonomy groupoid by $G_H$. The unit space of $G_H$ is the compact manifold $M$, the morphism space is the set of holonomy classes of curves along leaves of $(M,F)$. The range map and the source map $r,s: G_H \to M$ are given by sending curves to their terminal and initial points respectively. Groupoid multiplications are given by concatenation of curves.

\begin{proposition}[{\cite[Prop~5.6]{MoerdijkMrcun03}}]\label{prop-smooth-structure-holonomy}
	The morphism space of holonomy groupoid $G_H$ has a manifold structure.
\end{proposition}

\begin{proof}
	Let $\gamma \in G_H$ be some curves in a leaf of the foliation $(M,F)$. We shall construct an open  neighborhood of $\gamma$ which is homeomorphic to some Euclidean space. 
	
	Assume that $r(\gamma)=x$ and $s(\gamma)=y$. Pick local foliation charts $x\in U=T_1\times L_1 \to \mathbb{R}^p\times \mathbb{R}^q$ with $x=(x_T,x_L)\in U$ and $y\in V=T_2\times L_2\to \mathbb{R}^p \times \mathbb{R}^q$ with $y=(y_T,y_L)\in V$. If we pick two foliation charts small enough, there is a smooth map 
	\begin{equation}\label{eq-map-H}
	H: T_1\times [0,1] \to T_2
	\end{equation}
	such that $H(x_T,t)=\gamma(t)$ and $H(\ast,t)$ is a curve within some leaves connecting $H(\ast,0)\in T_1$ and $H(\ast,1)\in T_2$. Now we can define a map 
\begin{equation}\label{eq-smooth-structure-holonomy-groupoid}	
	T_1\times L_1\times L_2\to \holonomy(M,F)	
\end{equation}
	which assign $(a, b, c)\in T_1\times L_1\times L_2$ the curve $\tau\circ H(a,t) \circ \eta$ where $\tau$ is any curve connecting $(a,b)$ to $(a,x_L)$ in $U$, $H$ is the smooth map described in \eqref{eq-map-H} with $H(a,0)=(a,x_L)$ and $\eta$ is any curve connecting $H(a,1)$ with $(y_T, c)$ in V. The map \eqref{eq-smooth-structure-holonomy-groupoid} is well-defined since the holonomy class of $\tau\circ H(a,t) \circ \eta$ is independent of the choice of $\tau, H, \eta$. It is also clear that \eqref{eq-smooth-structure-holonomy-groupoid} is injective, and form a topological basis. In this way, we define the local Euclidean structure, and hence, a manifold structure of $G_H$.
\end{proof}

\begin{remark}
	As for the monodromy groupoid $G_M$, the unit space  is given by $M$, the morphism space is the set of homotopy classes of curves along leaves of $(M,F)$, the manifold structure and the source and range maps are given in a similar way.
\end{remark}

	The source fibers of $G_H$ and $G_M$ over $x\in M$ is the holonomy cover and universal cover of the leaf passing through $x\in M$ respectively.

\begin{example}
	If $F=TM$, the holonomy groupoid degenerates into the pair groupoid, namely $G_H = M\times M$. Under the same assumption $F=TM$, the morphism space of the monodromy groupoid is given by the space of homotopy classes of all curves in $M$. In this particular case, $G_M$ is usually called fundamental groupoid and  denoted by $\Pi(M)\rightrightarrows M$. It can be shown that the fundamental groupoid is Morita equivalent to fundamental group taken as groupoid over a single point. Hence, their corresponding groupoid $C^\ast$-algebras are Morita equivalent.
\end{example}

\section{Groupoid $C^\ast$-algebras}\label{sec-groupoid-algebra}
Parallel to the notion of Haar measures on locally compact topological groups, there is a notion of Haar systems on Lie groupoids.

\begin{definition}\label{def-Haar-system}
	Let $G\rightrightarrows G^{(0)}$ be a Lie groupoid, a family of measures $\{\mu_x\}_{x\in G^0}$ is called a Haar system on $G$ if
	\begin{enumerate}
		\item 
		The measure $\mu_x$ is supported on the source fiber $G_x$;
		\item
		For any smooth compactly supported function $f$ on $G$, the function on the unit space $G^{(0)}$ given by the assignment 
		\begin{equation*}
			x\mapsto \int_{G_x} f(\gamma) d\mu_x(\gamma)
		\end{equation*}
		is smooth;
		\item
		Let $\eta\in G$, $f$ be any smooth compactly supported function on $G$, then
		\begin{equation*}
			\int_{G_{s(\eta)}} f(\gamma) d\mu_{s(\eta)}(\gamma) = \int_{G_{r(\eta)}} f(\gamma\circ \eta) d\mu_{r(\eta)}(\gamma).
		\end{equation*}
	\end{enumerate}
	The above family of measures is sometimes referred to as right invariant Haar system. Left invariant Haar system $\{\mu^x\}_{x\in G^0}$ can be defined in a similar way where we replace the source fibers $G_x$ with the range fibers $G^x$.
\end{definition}

\begin{example}
	Recall that $(M,F)$ is a compact foliation. Fix a metric on $F$, then there is an induced measure $\{\mu_x\}_{x\in M}$ on each leaf $L_x\subset M$. The leafwise measures, in turn, determine measures $\{\mu_x^H\}$ on their holonomy covers $G_{H,x}$ and measures $\{\mu_x^M\}$ on universal covers $G_{M,x}$ accordingly. One can check that these measures form Haar systems on $G_H$ and $G_M$ respectively.
\end{example}

At the presence of a Haar system $\{\mu_x\}_{x\in G^{(0)}}$, the space of compactly supported smooth functions on groupoid $G$ can be made into an algebra. Let $f,g\in C^\infty_c(G)$, the multiplication $f\ast g$ is given by
\begin{equation}\label{eq-multiplication-of-groupoid-algebra}
	f\ast g (\gamma) = \int_{\gamma_1 \in G_{s(\gamma)}} f(\gamma\circ \gamma_1^{-1}) g(\gamma_1) d\mu_{s(\gamma)}(\gamma_1)
\end{equation}
and the adjoint is given by 
\begin{equation}
	f^\ast(\gamma) = \overline{f(\gamma^{-1})}.
\end{equation}

\begin{remark}
	In general, the monodromy groupoid $G_M$ and the holonomy groupoid $G_H$ may not be Hausdorff. We need to be careful with the definition of $C^\infty_c(G)$. Since $G_M$ and $G_H$ all have smooth manifold structure, every point in the groupoid have Hausdorff local coordinate chart. According to \cite{Connes80}, the space $C^\infty_c(G)$ is defined to be the span of functions each of which is smooth on a Hausdorff chart of $G$ and vanishes outside a compact subset of the Hausdorff chart. More precisely, a typical function in $C^\infty_c(G)$ can be written as finite sum
	\[
	f=\sum_i f_i,
	\]
	where $f_i$ is smooth function on a Hausdorff chart $U_i\subset G$ that vanishes outside a compact subset of $U_i$. If $G$ is indeed Hausdorff, then so defined $C^\infty_c(G)$ has its usual meaning (see \cite{Paterson99} for more detail).
\end{remark}

\begin{definition}
	Let $f\in C^\infty_c(G)$ and define
\begin{equation*}
	||f||_I = \sup_{x\in G^{(0)}}\left\{ \int_{G_x} |f(\gamma)|d\mu_x(\gamma), \int_{G_x} |f(\gamma^{-1})|d\mu_x(\gamma)\right\}.
\end{equation*}
It is easy to check that $||\cdot ||_I$ is a norm. We shall say a representation $\varphi: C^\infty_c(G) \to \mathcal{B}(\mathcal{H}_{\varphi})$ is bounded if it satisfies
\begin{equation*}
	||\varphi(f)||_{\mathcal{B}(\mathcal{H}_{\varphi})} \leq ||f||_I
\end{equation*}
for all $f\in C^\infty_c(G)$. The full groupoid $C^\ast$-algebra is the completion of $C^\infty_c(G)$ with respect to the norm
\begin{equation*}
	\sup_\varphi ||\varphi(f)||_{\mathcal{B(H_\varphi)}},
\end{equation*}
where $\varphi$ ranges over all bounded representations of $C^\infty_c(G)$. The full groupoid $C^\ast$-algebra is usually  denoted by $C^\ast G$.
\end{definition}

Analogous to the fact that the holonomy group at a fixed point is a quotient of the fundamental group of the leaf passing through the fixed point, there is a canonical quotient map $\pi: G_M \to G_H$ which sends the universal cover of a leaf to its holonomy cover. There is a homomorphism of algebras $\Phi: C^\infty_c(G_M) \to C^\infty_c(G_H)$ which is given by
\[
\Phi(f)(\eta) = \sum_{\pi(\gamma)=\eta} f(\gamma).
\]
\begin{proposition}\label{prop-canonical-map-monodromy-to-holonomy}
	The map $\Phi$ extends to a $C^\ast$-algebras homomorphism $C^\ast G_M \to C^\ast G_H$.
\end{proposition}

\begin{proof}
It is straightforward to check that $||\Phi(f)||_{C^\ast G_H} \leq ||\Phi(f)||_I \leq ||f||_I$.
\end{proof}

Within the set of bounded representations of groupoid algebra $C^\infty_c(G)$ there is a distinguished one called regular representation which is described as follows. Let $\{\mu_x\}$ be a right Haar system on the groupoid $G$. For any $x\in G^{(0)}$, the groupoid algebra $C^\infty_c(G)$ acts on the Hilbert space $L^2(G_x, \mu_x)$ as follow
\begin{equation*}
	\pi_x(f)\xi (\gamma) = \int_{\eta\in G_{x}} f(\gamma\circ \eta^{-1}) \xi(\eta) d\mu_x(\eta).
\end{equation*}
It is easy to check that this is a bounded representation. The completion of $C^\infty_c(G)$ with respect to the norm 
\begin{equation*}
	||f|| = \sup_{x\in G^{(0)}} ||\pi_x(f)||
\end{equation*}
is denoted by $C^\ast_r G$ and called the reduced groupoid $C^\ast$-algebra. By definition $||\cdot ||_{C^\ast_r G} \leq ||\cdot ||_{C^\ast G}$, so there is a canonical map $C^\ast G \to C^\ast_r G$.

Following the construction of groupoid $C^\ast$-algebra, we shall consider a construction of crossed product $C^\ast$-algebra. This algebra will be useful in the following sections. Let $B$ be a $C^\ast$-algebra, notice that the $C_c(G,B)$ has a $\ast$-algebra structure whose multiplication is given in the same way as in \eqref{eq-multiplication-of-groupoid-algebra} and the adjoint is given by $f^\ast(\gamma) = f(\gamma^{-1})^\ast$.
\begin{definition}
	Let $B$ be a $C^\ast$-algebra, let $f\in C_c(G, B)$, define a norm $||\cdot||_I$ on $C_c(G,B)$:
	\begin{equation*}
		||f||_I = \sup_{x\in G^{(0)}}\left\{ \int_{G_x} ||f(\gamma)||_Bd\mu_x(\gamma), \int_{G_x} ||f(\gamma^{-1})||_Bd\mu_x(\gamma)\right\}.
	\end{equation*}
	A representation $\varphi: C_c(G,B)\to \mathcal{B}(\mathcal{H}_\varphi)$ is called bounded if 
	\[
	||\varphi(f)||_{\mathcal{B}(\mathcal{H}_\varphi)} \leq ||f||_I
	\]
	for all $f\in C_c(G,B)$.
	The $C^\ast$-algebra $C^\ast(G,B)$ is defined to be the completion of $C_c(G, B)$ with respect to the norm
	\[
	||f||_{C^\ast(G,B)}= \sup_\varphi ||\varphi(f)||_{\mathcal{B}(\mathcal{H}_\varphi)},
	\]
	where $\varphi$ ranges over all bounded representations.
\end{definition}

\begin{proposition}
	Let $B\to C$ be a homomorphism between $C^\ast$-algebras , then it induces a homomorphism  $C^\ast (G,B) \to C^\ast (G,C)$.
\end{proposition}

\begin{proof}
	It is clear that the homomorphism $B \to C$ induces a $\ast$-homomorphism at the level of continuous maps $\pi: C_c(G,B) \to C_c(G,C)$. Let $\varphi: C_c(G,C) \to \mathcal{B}(\mathcal{H}_\varphi)$ be any bounded representation. Then the composition
	\begin{equation}\label{eq-composition-bounded-repn}
	C_c(G,B) \to C_c(G,C) \to \mathcal{B}(\mathcal{H}_\varphi)
	\end{equation}
	is also a representation. For any $f\in C_c(G,B)$, we have the following estimate
	\[
	||\varphi(\pi(f))||_{\mathcal{B}(\mathcal{H}_\varphi)} \leq ||\pi(f)||_I \leq ||f||_I,
	\]
	where the first inequality is a consequence of the boundedness assumption on $\varphi$ and the second inequality is implied by the fact that homomorphism between $C^\ast$-algebra is contractive. Therefore the composition \eqref{eq-composition-bounded-repn} is still a bounded representation of $C_c(G,B)$ and $||f||_{C^\ast (G,B)}\geq ||\pi(f)||_{C^\ast (G,C)}$ for all $f\in C_c(G.B)$. This completes the proof.
\end{proof}

\begin{proposition}\label{prop-exact-sequence-of-new-C-algebra}
	Let $B$ be a $C^\ast$-algebra, $J\subset B$ an ideal. If $G^{(0)}$ is compact, the exact sequence of $C^\ast$-algebras $0\to J\to B\to B/J\to 0$ induces an exact sequence
	\begin{equation}\label{eq-exact-sequence-algebra}
	0\to C^\ast(G,J)\to  C^\ast(G,B)\to C^\ast(G,B/J)\to 0.
	\end{equation}
\end{proposition}

\begin{proof}
	
	We first notice that the sequence at the level of continuous maps 
	\[
	0\to C_c(G,J) \to C_c(G,B) \to C_c(G,B/J) \to 0
	\]
	is exact. Indeed, the injectivity of the second arrow and the exactness in the middle term is clear. We only need to show the surjectivity of the third arrow. Pick a $f\in C_c(G,B/J)$ whose support is a compact subset $K$ of a coordinate chart $U\subset G$. There is an open neighborhood $U^\prime$ of $K$ such that the closure of $U^\prime$ is contained in $U$ and is compact. Then $f\in C_0(U^\prime)\otimes B/J$. Since $C_0(U^\prime)\otimes B\to C_0(U^\prime)\otimes B/J$ is surjective (see \cite[Sec~3.7]{NateOzawa08} for example), there is a preimage in $C_0(U^\prime)\otimes B \subset C_c(U,B) \subset C_c(G,B)$. General elements in $C_c(G,B)$ are spanned by those $f$'s. This proves the surjectivity of $C_c(G,B) \to C_c(G,B/J)$.
	
	It is clear that any bounded representation of $C_c(G,B)$ restricts to a bounded representation of $C_c(G,J)$. The $C^\ast$-norm $||\cdot||_{C^\ast(G,J)}$ on $C_c(G,J)$ is greater than or equal to the restriction of $||\cdot||_{C^\ast (G,B)}$ to $C_c(G,J)$. To show the injectivity of the second arrow in \eqref{eq-exact-sequence-algebra}, it suffice to show that any bounded representation of $C_c(G,J)$ extends to a bounded representation of $C_c(G,B)$.
	Indeed, let $\varphi: C_c(G,J)\to \mathcal{B}(\mathcal{H}_\varphi)$ be a bounded representation of $C_c(G,J)$, let $$\mathcal{H}^\prime=\text{closure of }\operatorname{span}\left\{\varphi(f)h \mid f\in C_c(G,J), h\in \mathcal{H}_\varphi\right\} \subset \mathcal{H}_\varphi$$
	be the Hilbert subspace of $\mathcal{H}_\varphi$. The algebra $C_c(G,B)$ acts on $\mathcal{H}^\prime$ in the following way:
	\begin{equation}\label{eq-representation-from-ideal}
	g.\varphi(f)h=\varphi(gf)h
	\end{equation}
	for all $g\in C_c(G,B)$ and $f\in C_c(G,J)$. To proceed, we need the following lemma.
	\begin{lemma}\label{lem-estimate-power}
		The representation \eqref{eq-representation-from-ideal} is bounded.
	\end{lemma}
	\begin{proof}
		Let $g\in C_c(G,B)$ and $f\in C_c(G,J)$. Since $\varphi$ is a bounded representation, we have $||\varphi(gf)||\leq ||g||_I\cdot ||f||_I$. Moreover, $||\varphi(gf)||^2 = ||\varphi(f^\ast g^\ast) \varphi(gf)||\leq ||\varphi(f)||\cdot ||g||_I^2 \cdot ||f||_I$. By induction, we have $||\varphi(gf)||^{2^k} \leq ||\varphi(f)||^{2^k-1}\cdot ||g||^{2^k}_I\cdot ||f||_I$ for all integers $k$. Taking the $2^k$-th root, we have 
		$
		||\varphi(gf)||\leq ||\varphi(f)||^{1-2^{-k}}\cdot ||g||_I\cdot ||f||_I^{2^{-k}}
		$
		for all $k\in \mathbb{N}$. Let $k\to \infty$, we get 
		\begin{equation}\label{eq-estimate-of-representation}
		||\varphi(gf)||\leq ||g||_I\cdot ||\varphi(f)||.
		\end{equation}
		Let $\{e_i\}$ be norm $1$ approximate identity of $C^\ast(G,J)$. Choose a sequence $\{v_i\}$  from $C_c(G,J)$ such that $||v_i-e_i||_{C^\ast (G,J)}\leq 1/i$. Then according to \eqref{eq-estimate-of-representation}, we have
		\[
		\varphi(gf)v=\lim_{i\to \infty} \varphi(gv_if)v.
		\]
		Moreover, 
		\[
		||\varphi(gv_if)v||\leq ||\varphi(gv_i)||\cdot ||\varphi(f)v||_{{\mathcal{H}_\varphi}}\leq ||g||_I\cdot ||\varphi(v_i)|| \cdot ||\varphi(f)v||_{\mathcal{H}_\varphi}.
		\]
		Taking the limit $i\to \infty$, we have 
		\[
		||\varphi(gf)v||_{\mathcal{H}_\varphi}\leq ||g||_I\cdot ||\varphi(f)v||_{\mathcal{H}_\varphi}.
		\]
		Since elements of the form $\varphi(f)v$ form a dense subspace of $\mathcal{H}^\prime$, the above estimate completes the proof of the lemma. 
	\end{proof}
	
	This shows that for any $f\in C_c(G,J)$ we have $||f||_{C^\ast(G,J)} =||f||_{C^\ast(G,B)}$. Hence the second arrow of \eqref{eq-exact-sequence-algebra} is injective. Since the range of homomorphism between $C^\ast$-algebras is close, the third map of \eqref{eq-exact-sequence-algebra} is surjective. It remains to show the exactness in the middle of \eqref{eq-exact-sequence-algebra}.
	
	A priori, the sequence \eqref{eq-exact-sequence-algebra} is only a complex, namely the composition of the second arrow and the third arrow is zero in \eqref{eq-exact-sequence-algebra}. There is a quotient map
	\begin{equation}\label{eq-quotient-isomorphism}
	C^\ast(G,B)/C^\ast(G,J)\to C^\ast(G,B/J).
	\end{equation}
	On the other hand, $C_c(G,B)/C_c(G,J)\cong C_c(G,B/J)$ sits inside $C^\ast(G,B/J)$. So there is a dense embedding 
	\begin{equation}\label{eq-dense-embedding}
	C_c(G,B/J)\hookrightarrow C^\ast(G,B)/C^\ast(G,J)
	\end{equation} 
	of algebras. Pick any faithful representation $\Psi: C^\ast(G,B)/C^\ast(G,J)\to \mathcal{B}(\mathcal{H}_\Psi)$, then the composition with \eqref{eq-dense-embedding} gives a representation $\pi$ of $C_c(G,B/J)$. We shall now prove that $\pi$ is a bounded representation. Let $f\in C_c(G,B/J)$ and $\bar{f}\in C_c(G,B)$ be a lift of $f$. Then we have 
	\begin{align*}
	||\pi(f)||_{\mathcal{B}(\mathcal{H}_\Psi)} &= ||f||_{C^\ast(G,B)/C^\ast(G,J)}\\
	&=\inf_{h\in C^\ast(G,J)} ||\bar{f}+h||_{C^\ast (G,B)}.
	\end{align*}
	Let $\{v_i\}$ be approximate identity of $J$ such that $0<v_i\leq v_j<1$ in the unitalization of $J$ if $i\leq j$. Then for any $h\in C_c(G,J)$, we have $h^\ast(x)(1-v_i)h(x)\geq h^\ast(x)(1-v_j)h(x)$ if $i\leq j$ which implies $||(1-v_i)^{1/2}h(x)||_J\geq ||(1-v_j)^{1/2}h(x)||_J$ if $i\leq j$. Therefore, the function
	\[
	g_i(u)=\int_{G^u} ||(1-v_i)^{1/2}h(x)||_J dx
	\]
	is continuous in $u$ and decreasing in $i$. According to the Monotone convergence theorem, the function $g_i(u)$ pointwise converge to zero function. On the other hand, since $G^{(0)}$ is compact, according to the Dini theorem, $g_i(u)$ converges uniformly to zero function. This implies that $(1-v_i)^{1/2}f \to 0$ in $I$-norm and also in the norm of $C^\ast(G,J)$. To proceed, we need the following lemma.
	
	\begin{lemma}
		$||f||_{C^\ast (G,B)/C^\ast (G,J)} = \lim_{i\to \infty} ||(1-v_i)^{1/2}\bar{f}||_{C^\ast (G,B)}$.
	\end{lemma}

	\begin{proof}
		Fix $\varepsilon>0$, there is $h\in C_c(G, J)$ such that 
		\[
		||\bar{f}-h||_{C^\ast (G, B)}\leq ||f||_{C^\ast (G,B)/C^\ast (G,J)}+\varepsilon.
		\]
		Then
		\[
		 ||(1-v_i)^{1/2}\bar{f}||_{C^\ast (G,B)}\leq  ||(1-v_i)^{1/2}(\bar{f}-h)||_{C^\ast (G,B)}+ ||(1-v_i)^{1/2}h||_{C^\ast (G,B)}.
		\]
		Using the method of Lemma~\ref{lem-estimate-power}, we can show that $||(1-v_i)^{1/2}(\bar{f}-h)||_{C^\ast (G,B)}\leq ||(1-v_i)^{1/2}||\cdot ||\bar{f}-h||_{C^\ast (G,B)}.$ Therefore, choosing $i$ sufficiently large we can arrange that 
		\[
		||(1-v_i)^{1/2}\bar{f}||_{C^\ast (G,B)}\leq ||f||_{C^\ast (G,B)/C^\ast (G,J)}+2\varepsilon.
		\] 
		This completes the proof of the lemma.
	\end{proof}

	Thanks to the above lemma, we have $||\pi(f)||_{\mathcal{B}(\mathcal{H}_\Psi)}\leq \lim_{i\to \infty} ||(1-v_i)^{1/2}\bar{f}||_I$. Again, let
	\[
	g_i(u) = \int_{G^u} ||(1-v_i)^{1/2}\bar{f}(x)||_B dx
	\]
and apply the Dini theorem once again, we have $g_i(u)$ is uniformly convergent in $u$ and 
\[
\lim_{i\to \infty} \sup_u g_i(u) = \sup_u \int_{G^u} ||f||_{B/J}. 
\]
Overall, we have $||\pi(f)||_{\mathcal{B}(\mathcal{H}_\Psi)}\leq ||f||_I$.
	So $\pi$ is a bounded representation. Therefore the norm on $C^\ast(G,B)/C^\ast(G,J)$ is less than or equal to the norm on $C^\ast(G,B/J)$. Together with the fact that homomorphism between $C^\ast$-algebra is contractive, we have \eqref{eq-quotient-isomorphism} is an isomorphism. This completes the proof.
\end{proof}

\section{Rosenberg Index}\label{sec-rosenberg-index}

From this section on, except the last section, we shall assume that $F\to M$ is a spin vector bundle of even dimensional with spinor given by $S=S^+(F)\oplus S^-(F)$. Let $D_+$ be the positive part of the leafwise Dirac operator acting on $S$, which means the following
\begin{itemize}
	\item $D_+: C^\infty(M,S^{+}) \to C^\infty(M,S^{-})$ is a usual differential operator;
	\item For any smooth section $\xi$ of $S^+\to M$ and any leaf $L$ of $(M,F)$, the restriction $D_+ \xi|_L$ only depends on the restriction $\xi|_L$;
	\item For any leaf $L$ of $M$, $D_+|_L: C^\infty_c(L,S^{+}) \to C^\infty_c(L,S^{-})$ is the classical Dirac operator on $L$.
\end{itemize}
The notion of leafwise Dirac type operator or more generally, the notion of leafwise elliptic differential operator can be defined in a similar way (see \cite[Ch~2, Sec~9]{Connes94}). 

The Dirac operator $D_{+}|_L$ can be lifted to  universal covers $D_{+, \widetilde{L}}: C^\infty_c(\widetilde{L}, \pi^\ast S^{+})\to C^\infty_c(\widetilde{L}, \pi^\ast S^{-})$ where $\pi: \widetilde{L}\to L$ is the covering map. All those $D_{+,\widetilde{L}}$'s can be assembled to an operator $D_+: C^\infty_c(G_M, r^\ast S^{+})\to C^\infty_c(G_M, r^\ast S^{-})$ such that
\[
D_{+}f(\gamma)=\left(D_{+,\widetilde{L}_{s(\gamma)}} f\right) (\gamma),
\]
where $f\in C^\infty_c(G_M, r^\ast S^{+})$ and $\widetilde{L}_{s(\gamma)}$ is the universal cover of the leaf passing through $s(\gamma)\in M$. Similarly, the operator $D_{-}:  C^\infty_c(G_M, r^\ast S^{-}) \to C^\infty_c(G_M, r^\ast S^{+})$ can be defined.

\begin{proposition}\label{prop-complete-section-spinor-module}
The space $C^\infty_c(G_M,r^\ast S)$ can be completed into a Hilbert $C^\ast G_M$-module which will be denoted by $\mathcal{E}$. 
\end{proposition}

\begin{proof}
	Let $\varphi, \psi \in C^\infty_c(G_M, r^\ast S)$ and $\gamma\in G_M$, the formula
	\begin{equation*}
		\langle \varphi, \psi \rangle(\gamma) = \int_{\gamma_1 \in G_{M, s(\gamma)}} \langle \varphi(\gamma_1\circ \gamma^{-1}), \psi(\gamma_1) \rangle d\mu_{s(\gamma)}(\gamma_1)
	\end{equation*}
defines a $C^\ast G_M$-valued inner product on $C^\infty_c(G_M, r^\ast S)$. Let $f\in C^\infty_c(G_M)$, it acts on $C^\infty_c(G_M, r^\ast S)$ by
	\begin{equation*}
		\varphi.f(\gamma) = \int_{\gamma_1\in G_{M, s(\gamma)}} \varphi(\gamma \circ \gamma_1^{-1})f(\gamma_1) d\mu_{s(\gamma)}(\gamma_1).
	\end{equation*}
It is easy to check that this inner product satisfies the pre-Hilbert module condition and $\langle \varphi, \varphi \rangle=0$ implies $\varphi=0$. The completion of $C^\infty_c(G_M, r^\ast S)$ under the norm
\begin{equation*}
	||\varphi||_{\mathcal{E}}^2 = ||\langle \varphi, \varphi \rangle||_{C^\ast G_M}
\end{equation*}
is a Hilbert $C^\ast G_M$-module.
\end{proof}

The same constructions can be done for $S^+$ and $S^-$ and the corresponding Hilbert $C^\ast G_M$-modules will be denoted by $\mathcal{E}_+$ and $\mathcal{E}_-$ respectively. Clearly, $\mathcal{E}=\mathcal{E}_+\oplus \mathcal{E}_-$.

\begin{proposition}
		The operators $D_{+}$ and $D_{-}$ are formal adjoint to each other. Namely for any $f\in C^\infty_c(G_M, r^\ast S^{+})$ and $g\in C^\infty_c(G_M, r^\ast S^{-})$, we have
		\[
		\langle D_+ f, g\rangle = \langle f, D_-g\rangle.
		\]
\end{proposition}

\begin{proof}
	Indeed,
	\begin{align*}
	\langle D_+f, g \rangle(\gamma) &= \int_{\gamma_1\in G_{M,s(\gamma)}} \langle D_+f(\gamma_1\circ \gamma^{-1}), g(\gamma_1) \rangle d\mu_{s(\gamma)}(\gamma_1) \\
	&= \int_{\gamma_1\in G_{M,s(\gamma)}} \langle \left(D_{+, \widetilde{L}_{r(\gamma)}} f\right)(\gamma_1\circ \gamma^{-1}), g(\gamma_1) \rangle d\mu_{s(\gamma)}(\gamma_1)\\
	&=\langle D_{+, \widetilde{L}_{s(\gamma)}}(U_{\gamma^{-1}}f), g \rangle = \langle U_{\gamma^{-1}}f, D_{-, \widetilde{L}_{s(\gamma)}}g \rangle = \langle f, D_-g\rangle(\gamma),
	\end{align*}
	where $U_\gamma$ is the translation operator $U_\gamma f(\gamma_1)=f(\gamma_1\circ \gamma)$. In the last line, the first two inner products are given by the $L^2$ inner product of the space $C_c^\infty(\widetilde{L}_{s(\gamma)}, \pi^\ast S)$. The first two terms in the last line are the same because $D_+$ is formal adjoint to $D_-$ on the universal cover of leaves.
\end{proof}

In the following discussion, we shall use $D_+, D_-$ for their closure. According to \cite[Prop~21, Lem~22]{Vassout06}, $D_+$ and $D_-$ can be taken as unbounded regular operators and $D_-^\ast =D_+$. So
\[
D=
\begin{bmatrix}
	0 & D_-\\
	D_+ & 0
\end{bmatrix}
:
\mathcal{E} \to \mathcal{E}
\]
is self-adjoint and regular. Since $D\pm iI$ is a first order elliptic operator, there is a smoothing operator $R$ and pseudodifferential operator of order negative one $Q$ such that 
\[
(D\pm iI)Q = I+R.
\]
Multiply $(D\pm iI)^{-1}$ on both sides, we get $(D\pm iI)^{-1}$ is compact. So the functional calculus $f(D)$ (see \cite{Kucerovsky02} for example) is compact for all $f\in C_0(\mathbb{R})$.

Recall that in \cite[Chp~10]{HigsonRoe00}, a continuous function $f:\mathbb{R}\to [-1,1]$ is called normalizing if 
\begin{itemize}
	\item 
	$f$ is odd;
	\item
	$f(c)\geq 0$ if $c\geq 0$;
	\item
	$\lim_{c\to \pm \infty} f(c) \to \pm 1$.
\end{itemize}

\begin{definition}\label{def-rosenberg-index}
The Rosenberg index of $D$ is an element $[\alpha]$ in $K_0(C^\ast G_M)$ which is given by the Kasparov module $(\mathcal{E},f(D))$ for any normalizing function $f$.
\end{definition}

\begin{proposition}
	If $F$ is spin and $(M,F)$ admits leafwise positive scalar curvature, then $[\alpha]=0$ as a $K$-theory element in $K_0(C^\ast G_M)$. 
\end{proposition}

\begin{proof}
	By Lichnerowicz formula, if $(M,F)$ has leafwise positive scalar curvature, the leafwise Dirac $D$ is invertible. It has a spectrum gap around $0\in \mathbb{R}$. We can choose the normalizing function $f$ such that $f^2=1$ on spectrum of $D$. Under this circumstances, the Kasparov module $(\mathcal{E},f(D))$ is degenerated.
\end{proof}

\begin{remark}\label{remark-relation-between-rosenberg-index-and-longitudinal-index}
	In \cite{ConnesSkandalis84}, the authors define the longitudinal index as an element in $K_0(C^\ast_r G_H)$. Following their method, we set the Rosenberg index to live in $K_0(C^\ast G_M)$. In fact, under the map 
	\begin{equation}\label{eq-hom-from-monodromy-to-reduced-holonomy}
		C^\ast  G_M \to C^\ast G_H \to C^\ast_r G_H,
	\end{equation}
	the Rosenberg index defined above is mapped to the longitudinal index.
\end{remark}

\section{Twisted Rosenberg index}\label{sec-twisted-rosenberg-index}
In this section, we assume $B$ to be a $C^\ast$-algebra with unit. The theory of pseudodifferential operators over unital $C^\ast$-algebras can be found in \cite{MiscenkoFomenko79}. In \cite{MiscenkoFomenko79} the author define pseudodifferential operators over unital $C^\ast$-algebras for compact smooth manifolds, the method there also works for paracompact manifold. One can choose locally finite partition of unity in the formula (3.12) in \cite{MiscenkoFomenko79}.



Let $S^m(A^\ast G,B)$ be the set of all $a\in C^\infty(A^\ast G,B)$ such that for every compact subset $K\subset G^0$ and every multi-indices $\alpha, \beta$ there is constant $C_{\alpha, \beta, K}>0$ with the following inequality
\[
||\partial_x^\alpha \partial_\xi^\beta a(x,\xi)||_B \leq C_{\alpha, \beta, K}\cdot (1+|\xi|)^{m-|\beta|},
\]
for all $x\in K$. Let $S^m_{\operatorname{phg}}(A^\ast G,B)$ be the set of all $a\in S^m(A^\ast G,B)$ such that for every $j\in \mathbb{N}$ one can find $a_{m-j}\in C^\infty(A^\ast G,B)$ with the property $a_{m-j}(x,t\xi) = t^{m-j} a(x,\xi)$ for all $t>0$, $||\xi||\geq 1$ and 
\[
a-\sum_{j=0}^{N-1} a_{m-j} \in S^{m-N}(A^\ast G,B),
\]
for all $N\in \mathbb{N}$. 

\begin{definition}
	A pseudodifferential operator of order $m$ on Lie groupoid $G$ with values in $B$ is a compactly supported $G$-operator $\{P_x\}_{x\in G^{(0)}}$ in the sense of \cite[sec~3.3]{Vassout06} such that 
	\begin{itemize}
		\item 
		each $P_x$ is a pseudodifferential operator on source fiber $s^{-1}(x)$ of order $m$ over a $C^\ast$-algebra $B$;
		\item
		for each trivializing open subset $U\times V \cong \Omega\subset G$ to which the source map restricts to the projection onto the first factor, and for all $\phi, \psi \in C_c(\Omega)$ the operator $\phi P_x \psi$ is given by a symbol $a(x,y,\xi)\in S^m_{\operatorname{phg}}(U\times V\times \mathbb{R}^n, B)$.
	\end{itemize}
If in addition, the distributional kernel of $\{P_x\}$ is compactly supported, the pseudodifferential operator is called compactly supported. The principal symbol $\sigma_P\in S^\ast_{\operatorname{phg}}(A^\ast G, B)$ of a pseudodifferential operator $P$ is defined by
\[
\sigma_P(x,\xi)= \sigma(P_x)(x,\xi),
\]
where $\sigma(P_x)\in S^\ast(T^\ast G_x, B)$ is the principal symbol of $P_x$ as pseudodifferential operator on the source fiber $G_x$. From the definition, it is clear that if $P,Q$ are compactly supported pseudodifferential operators, then $PQ$ is still a pseudodifferential operator and $\sigma_{PQ}=\sigma_P \cdot \sigma_Q$.
\end{definition}

\begin{proposition}
	Pseudodifferential operators on $G$ with compact support of order less than or equal to zero extend to morphisms between $C^\ast (G,B)$ and pseudodifferential operators with compact support of order strictly less than zero extend to elements of $C^\ast(G,B)$.
\end{proposition}
\begin{proof}
	(see \cite[Proposition~3.4]{SkandalisDebord2019lie})
	We first assume that the pseudodifferential operator $P$ has order less than or equal to $p=\dim G^{(0)} -\dim G$. Then for any trivializing open subset $U\times V \cong \Omega\subset G$ and any $\phi, \psi \in C_c(\Omega)$ the operator $\phi P_x \psi$ has smooth integral kernel. Therefore, $P$ has compactly supported smooth kernel which clearly extends to an element of $C^\ast(G,B)$.
	
	If $P$ has order $\leq p/2$, then $||Pf||^2_{C^\ast (G,B)}\leq ||\langle Pf, Pf\rangle||_{C^\ast (G,B)} \leq ||P^\ast P||\cdot ||f||^2_{C^\ast (G,B)}$ which implies that  $P$ is a multiplier of $C^\ast (G,B)$. Since $P^\ast P$ extends to an element of $C^\ast (G,B)$, it follows that $P\in C^\ast (G,B)$. By induction, if $P$ has order $\leq p/2^k$ for some integer $k$, then $P\in C^\ast (G,B)$. This proves compactly supported pseudodifferential operators of negative order extend to an element of $C^\ast (G,B)$.
	
	Now assume that $P$ is of order $0$ with principal symbol $\sigma_P\in S^0_{\operatorname{phg}}(A^\ast G, B)$. Let $c\in \mathbb{R}_+$ such that $c> \sigma_p(x,\xi)$ for all $(x,\xi)\in A^\ast G$. Put $b(x,\xi)=(c^2+1-|\sigma_p(x,\xi)|^2)^{1/2}$ and let $Q$ be pseudodifferential operator with principal symbol $b(x,\xi)$. Then $P^\ast P+Q^\ast Q$ has principal symbol $1+c^2$ and is bounded. A direct calculation 
	\[
	||Pf||^2_{C^\ast (G,B)}\leq ||\langle Pf, Pf\rangle||_{C^\ast (G,B)} \leq ||\langle (P^\ast P+ Q^\ast Q)f, f\rangle||_{C^\ast (G,B)}
	\]
	shows that $P$ is bounded. This completes the proof.
\end{proof}

Given $a\in S^m_{\operatorname{phg}}(A^\ast G, B)$, a pseudodifferential operator $P_a: C^\infty_c(G, B) \to C^\infty_c(G,B)$ can be defined by the formula in \cite[Prop~14]{Vassout06}. Namely, we fix a diffeomorphism $\phi$ from a tubular neighborhood of $G^{(0)}\subset G$ to an open neighborhood $W$ of the zero section of $AG$. Let $\chi$ be a function with values in $[0,1]$, whose restriction to $G^{(0)}$ equals $1$ and its support is contained in $W$. Let $\xi \in A^\ast G, \gamma\in G$, and $e_\xi(\gamma)=\chi(\gamma) \exp(i\langle \phi(\gamma), \xi\rangle)$, then $P_a$ is given by the distributional kernel 
\[
k(\gamma) = \frac{1}{(2\pi)^n} \int_{A^\ast_{r(\gamma)} G} e_{-\xi}(\gamma^{-1}) a(r(\gamma), \xi) d\xi.
\]
Then $P_a$ is a pseudodifferential operator on $G$ of order $m$ whose principal symbol is $a$.

\begin{remark}
	The above discussion also applies to pseudodifferential operators between finitely generated projective Hilbert $B$-module bundles.
\end{remark}

Let $G^{(0)}$ be compact. A pseudodifferential operator is called elliptic if its principal symbol $a\in S^m_{\operatorname{phg}}(A^\ast G, B)$ is invertible outside a compact neighborhood of the zero section $G^{(0)}\subset A^\ast G$. Then there is $a^\prime \in S^{-m}_{\operatorname{phg}}(A^\ast G, B)$ which agrees with the inverse of $a$ outside a compact neighborhood of $G^{(0)}\subset A^\ast G$.

Now, let $E$ be a finitely generated projective Hilbert $B$-module bundle over $M$.  The space $C^\infty_c(G_M, r^\ast E)$ can be completed into a Hilbert $C^\ast (G_M, B)$-module in the same way as Proposition~\ref{prop-complete-section-spinor-module}. Notice that $S^+\otimes E$ ($S^-\otimes E, S\otimes E$ respectively) is still a finitely generated projective Hilbert $B$-module bundle over $M$, the corresponding Hilbert module will be denoted by $\mathcal{E}_{+,B}$ ($\mathcal{E}_{-,B}, \mathcal{E}_{B}$ respectively). Notice that $\mathcal{E}_B=\mathcal{E}_{+,B}\oplus \mathcal{E}_{-,B}$. Let $D_{+, E}: C^\infty_c(G_M, r^\ast S^{+}\otimes r^\ast E)\to C^\infty_c(G_M, r^\ast S^{-}\otimes r^\ast E)$ denote the leafwise Dirac type operator twisted by $E$ which is a first order elliptic differential operator. The operator $D_{+,E}$ can be taken as an unbounded operator from $\mathcal{E}_{+,B}$ to $\mathcal{E}_{-,B}$. We shall use the same notation for its closure.

\begin{proposition}
	The operator $D_{+,E}: \mathcal{E}_{+,B} \to \mathcal{E}_{-,B}$ is regular and $D_{+, E}^\ast = D_{-,E}$.
\end{proposition}


\begin{proof}
	Since $D_{+,E}$ is elliptic, there is a pseudodifferential operator  $Q$ of order $-1$ such that $D_{+,E}Q-I=R$ and $QD_{+,E}-I=S$ are smoothing operators. Then the proof in \cite[Prop~21]{Vassout06} works verbatim.
\end{proof}

Let $D_E=\begin{bmatrix}
	0 & D_{-,E}\\
	D_{+,E} & 0
\end{bmatrix}
: \mathcal{E}_B \to \mathcal{E}_B
$. It is a self-adjoint regular operator. It can be checked that the operator $(D_E\pm i)^{-1}: \mathcal{E}_B \to \mathcal{E}_B$ is compact.

\begin{proposition}
	Let $f$ be a normalizing function, then the pair $\left(\mathcal{E}_B, f(D_E)\right)$ forms a Kasparov module and determines an element in $K_0(C^\ast (G_M, B))$. This element will be called twisted Rosenberg index and denoted by $[D_E]$.
\end{proposition} 

\begin{proof}
	It suffices to show that if $g$ vanishes at infinity, $g(D_E): \mathcal{E}_B\to \mathcal{E}_B$ is a compact operator. The result follows from the fact that $C_0(\mathbb{R})$ is generated by $(x\pm i)^{-1}$ as $C^\ast$-algebra.
\end{proof}

\section{The Hilbert module out of leafwise flat bundles}\label{sec-hilbert-module}

The basic theory of Hilbert $C^\ast$-module and Hilbert $C^\ast$-module bundle can be found in \cite{Schick05}. Let $B$ be an unital $C^\ast$-algebra, let $W$ be a leafwise flat, finitely generated projective Hilbert $B$-module bundle over $M$.

The space $C^\infty_c(G_M,r^\ast W)$ has a $C^\infty_c(G_M, B)\subset C^\ast (G_M, B)$-valued inner product given by
\begin{equation}\label{eq-action-of-function-on-section}
	\langle \varphi, \psi \rangle (\gamma)= \int_{\gamma_1\in G_{M,s(\gamma)}} \langle \varphi(\gamma_1\circ \gamma^{-1}), \psi(\gamma_1) \rangle d\mu_{s(\gamma)}(\gamma_1),
\end{equation}
where $\varphi,\psi\in C^\infty_c(G_M, r^\ast W)$.
Let $\mathcal{E}_W$ be the completion of $C^\infty_c(G_M,r^\ast W)$  under the norm $||\varphi||_{\mathcal{E}_W} = ||\langle \varphi, \varphi \rangle||_{C^\ast(G_M,B)}^{1/2}$. The space $C^\infty_c(G_M,r^\ast W)$ has an obvious right $C^\infty_c(G_M,B)$ action which is given by
\begin{equation}\label{eq-action-compactly-pre-hilbert}
	\varphi.f(\gamma) = \int_{\gamma_1\in G_{M,s(\gamma)}} \varphi(\gamma\circ \gamma_1^{-1}) f(\gamma_1) d\mu_{s(\gamma)}(\gamma_1),
\end{equation}
where $\varphi\in C^\infty_c(G_M,r^\ast W)$ and $f\in C^\infty_c(G_M, B)$. 
The action \eqref{eq-action-compactly-pre-hilbert} extends to a right $C^\ast (G_M, B)$ action on $\mathcal{E}_W$. As a consequence,  $\mathcal{E}_W$ has a Hilbert $C^\ast (G_M, B)$-module structure.

There is a left $C_0(M)$-action on  $C^\infty_c(G_M, r^\ast W)$ which is given by
\begin{equation}\label{eq-action-left-bdd-function-on-sections}
	h.\varphi (\gamma) = h(r(\gamma))\cdot \varphi(\gamma),
\end{equation}
for all $h\in C_0(M)$ and $\varphi\in C^\infty_c(G_M, r^\ast W)$.

\begin{proposition}
	For any $h\in C_0(M)$ and $\varphi\in C^\infty_c(G_M, r^\ast W)$, we have
	\begin{equation*}
		||h.\varphi|| \leq ||h||\cdot ||\varphi||_{\mathcal{E}_W},
	\end{equation*}
where $||h||$ is the sup-norm of $h$ in $C_0(M)$.
\end{proposition}

\begin{proof}
	We shall use the estimate in \cite[Lem~1.1.13]{Renault80}.
	Let $k\in C_0(M)$ be the function defined by
	\begin{equation*}
		k(m) = \left(||h||^2-|h(m)|^2 \right)^{1/2}.
	\end{equation*}
Then, it is easy to check the following
	\begin{align*}
		||h.\varphi||^2 &= ||\langle h.\varphi, h.\varphi \rangle|| \\
		&=\left|\left| ||h||^2\langle \varphi, \varphi\rangle - \langle k.\varphi, k.\varphi \rangle \right|\right|  \\
		&\leq ||h||^2 ||\langle \varphi, \varphi \rangle ||,
	\end{align*}
which completes the proof.
\end{proof}

\begin{corollary}(\cite[Proposition~2.1.14]{Renault80})\label{coro-action-base-on-groupoid}
	The action \eqref{eq-action-left-bdd-function-on-sections} extends to a $\ast$-homomorphism $C_0(M) \to \mathcal{L}(\mathcal{E}_W)$.
\end{corollary}

\begin{proof}
	Thanks to the above proposition, the action extends to the Hilbert module $\mathcal{E}_W$. It is a matter of direct calculation to check that it preserves the $\ast$-operation. 
\end{proof}

We shall show, in the rest of this section, the Hilbert module $\mathcal{E}_W$ determines a $KK$-theory element in $KK(C^\ast G_M, C^\ast( G_M, B))$. There is a left $C^\infty_c(G_M)$ action on $C^\infty_c(G_M, r^\ast W)$ which is given by
\begin{equation}\label{eq-action-groupoid-on-hilbert-module}
	f.\varphi(\gamma) = \int_{\gamma_1\in G_{M, s(\gamma)}} f(\gamma\circ \gamma_1^{-1}) (\gamma\circ \gamma_1^{-1}).\varphi(\gamma_1) d\mu_{s(\gamma)}(\gamma_1),
\end{equation}
where $(\gamma\circ\gamma_1^{-1}).\varphi(\gamma_1)$ is the image of $\varphi(\gamma_1)$ under the parallel translation along the curve  $\gamma \circ \gamma_1^{-1}$. Thanks to the leafwise flatness of $W$, this parallel translation is well defined. It is also convenient to have an alternative description of the action \eqref{eq-action-groupoid-on-hilbert-module}. 

Let $\mu=\{\mu_x\}$ be a right invariant Haar system on $G_M$. The inverse map $\iota: G_M\to G_M$ induces a left invariant Haar system which we denote by $\widetilde{\mu}$. The space $C^\infty_c(G_M)$ can be completed into Hilbert $C_0(M)$-modules in two ways given by two inner products
\begin{equation*}
	\langle f,g \rangle_s(m) = \int_{s(\gamma)=m} \overline{f(\gamma)}\cdot g(\gamma) d{\mu}(\gamma) 
\end{equation*}
and
\begin{equation*}
	\langle f,g \rangle_r(m) = \int_{r(\gamma)=m} \overline{f(\gamma)}\cdot g(\gamma) d\widetilde{\mu}(\gamma),
\end{equation*}
where $f,g \in C^\infty_c(G_M)$. Following \cite{BussHolkarMeyer18}, we shall denote the completions by $L^2(G_M, s,\mu)$ and $L^2(G_M,r,\widetilde{\mu})$ respectively. According to Corollary~\ref{coro-action-base-on-groupoid}, we can form the inner tensor product $L^2(G_M,s,\mu)\otimes_{C_0(M)} \mathcal{E}_W$ and $L^2(G_M,r,\widetilde{\mu})\otimes_{C_0(M)} \mathcal{E}_W$. We denote by $C^\infty_c(G_M)\otimes_{\operatorname{alg}} C^\infty_c(G_M,r^\ast W)$ the dense subset of $L^2(G_M,s,\mu)\otimes_{C_0(M)} \mathcal{E}_W$ consists of linear span of elements of the form $f\otimes \varphi$ with $f\in C^\infty_c(G_M)$ and $\varphi\in C^\infty_c(G_M, r^\ast W)$. 

\begin{proposition}\label{prop-approximation}
	There is a map $U: C^\infty_c(G_M)\otimes_{\operatorname{alg}} C^\infty_c(G_M,r^\ast W) \to L^2(G_M,r,\widetilde{\mu})\otimes_{C_0(M)} \mathcal{E}_W$ given by
\begin{equation}\label{eq-def-action-U}
	U(F)(\gamma_1,\gamma_2) = \gamma_1. F(\gamma_1, \gamma_1^{-1}\circ \gamma_2),
\end{equation}
where $\gamma_1.$ is the parallel translation of $W$ along $\gamma_1$. 
\end{proposition}
\begin{proof}
	We have to show that $U(F)\in L^2(G_M,r,\widetilde{\mu})\otimes_{C_0(M)} \mathcal{E}_W$. It is enough to verify the case where $F=f\otimes \varphi$. For $\gamma_1,\gamma_2\in G_M$ with $r(\gamma_1)=r(\gamma_2)$, we have
	\[
	U(F)(\gamma_1, \gamma_2) = f(\gamma_1)\cdot \gamma_1.\varphi(\gamma_1^{-1}\circ \gamma_2).
	\]
	By using the fact that $C^\infty_c(G_M, r^\ast W)$ is a finitely generated projective module over $C^\infty_c(G_M, B)$, the above equation can be written as a finite sum
	\[
	U(F)(\gamma_1, \gamma_2) = \sum_{i} F_{i}(\gamma_1, \gamma_2) \varphi_{i}(r(\gamma_2)),
	\]
	where $F_{i}$ are compactly supported smooth functions on $$H=\{(\gamma_1,\gamma_2)\in G_M\times G_M \mid r(\gamma_1)=r(\gamma_2)\}$$
	with values in $B$ and $\varphi_i$ are smooth compactly supported sections of $W\to M$.
	Since the image of $C^\infty_c(G_M)\otimes_{\operatorname{alg}} C^\infty_c(G_M,B)\to C^\infty_c(H,B)$ is dense in the inductive topology, there is a sequence  $F^k_i\in C^\infty_c(G_M)\otimes_{\operatorname{alg}} C^\infty_c(G_M,B)$ such that $F^k_i\to F_i$ in the inductive topology of $C^\infty_c(H, B)$ for all $i$. Therefore for any $\varepsilon>0$ there is $N\in \mathbb{N}$ such that 
	\[
	\left|\left|\sum_{i} (F^k_i-F^{k^\prime}_i)\cdot r^\ast \varphi_i\right|\right|_{L^2(G_M,s,\mu)\otimes_{C_0(M)} \mathcal{E}_W} \leq \varepsilon
	\]
	whenever $k,k^\prime >N$. As a consequence,
	$$\sum_{i} F^k_i\cdot  r^\ast \varphi_i\in C^\infty_c(G_M)\otimes_{\operatorname{alg}} C^\infty_c(G_M, r^\ast W)$$ is a Cauchy sequence parametrized by $k$ and converging to $U(F)$ in the topology of $L^2(G_M,s,\mu)\otimes_{C_0(M)} \mathcal{E}_W$. 
\end{proof}

\begin{proposition}
If $F,G$ either belong to $C^\infty_c(G_M)\otimes_{\operatorname{alg}} C^\infty_c(G_M,r^\ast W)$ or the image of $C^\infty_c(G_M)\otimes_{\operatorname{alg}} C^\infty_c(G_M,r^\ast W)$ under $U$, then $\langle F, G\rangle_{L^2(G_M,r,\widetilde{\mu})\otimes_{C_0(M)} \mathcal{E}_W}\in C^\infty_c(G_M, B)$ and
\begin{equation}\label{eq-calculation-formula}
\langle F, G\rangle_{L^2(G_M,r,\widetilde{\mu})\otimes_{C_0(M)} \mathcal{E}_W} (\gamma)= \int_{s(\gamma_1)=s(\gamma),r(\gamma_2)=r(\gamma_1)} \left\langle F(\gamma_2, \gamma_1\circ \gamma^{-1}), G(\gamma_2,\gamma_1) \right\rangle d\mu(\gamma_1) d\mu(\gamma_2).
\end{equation}
\end{proposition}
\begin{proof}
	If $F,G$ both belong to $C^\infty_c(G_M)\otimes_{\operatorname{alg}} C^\infty_c(G_M,r^\ast W)$ equation \eqref{eq-calculation-formula} is obvious. If at least one of them belongs to the image of $U$, then $F,G$ can be approximated by $F_i, G_i\in  C^\infty_c(G_M)\otimes_{\operatorname{alg}} C^\infty_c(G_M,r^\ast W)$ as constructed in Proposition~\ref{prop-approximation}. Moreover, $\left \langle F_i, G_i \right \rangle\to \langle F, G\rangle$  in the inductive topology of $C^\infty_c(G_M, B)$. This completes the proof.
\end{proof}

\begin{proposition}
	$$\langle U(F), U(F) \rangle_{ L^2(G_M,r,\widetilde{\mu})\otimes_{C_0(M)} \mathcal{E}_W} = \langle F, F\rangle_{ L^2(G_M,s,\mu)\otimes_{C_0(M)} \mathcal{E}_W},$$ for all $F\in C^\infty_c(G_M)\otimes_{\operatorname{alg}} C^\infty_c(G_M,r^\ast W)$.
\end{proposition}

\begin{proof}
	According to \eqref{eq-calculation-formula}, we have
	\begin{align*}
	\langle U(F), U(F) \rangle(\gamma) &= \int \left\langle \gamma_2.F(\gamma_2, \gamma_2^{-1}\circ \gamma_1\circ \gamma^{-1}), \gamma_2.F(\gamma_2, \gamma_2^{-1}\circ \gamma_1)\right\rangle d\mu(\gamma_1) d\mu(\gamma_2)\\
	&=\int \left\langle F(\gamma_2, \gamma_2^{-1}\circ \gamma_1\circ \gamma^{-1}), F(\gamma_2, \gamma_2^{-1}\circ \gamma_1)\right\rangle d\mu(\gamma_1) d\mu(\gamma_2)\\
	&=\langle F, F\rangle(\gamma),
	\end{align*}
where from the first line to the second line we use the fact that parallel translation is unitary and from the second line to the third line we use the right invariance of Haar system (See Definition~\ref{def-Haar-system}, iii).
\end{proof}

Therefore the map $U$ can be extended to an isometry 
\[
U: L^2(G_M,s,\mu)\otimes_{C_0(M)} \mathcal{E}_W \to L^2(G_M,r,\widetilde{\mu})\otimes_{C_0(M)} \mathcal{E}_W.
\] 
Let $f\in C^\infty_c(G_M)$, let $T_f: \mathcal{E}_W \to L^2(G_M,s,\mu)\otimes_{C_0(M)} \mathcal{E}_W$ denote the Hilbert module map $x\mapsto f\otimes x$ and $T^\ast_f: L^2(G_M, r,\widetilde{\mu})\otimes_{C_0(M)} \mathcal{E}_W\to \mathcal{E}_W$ be the operator which sends $g\otimes x$ to $\langle f,g \rangle_r.x$. Choose $f_1,f_2\in C^\infty_c(G_M)$ such that $f=\bar{f}_1\cdot f_2$, where $\cdot$ is the point-wise multiplication. 
\begin{proposition}
	The action \eqref{eq-action-groupoid-on-hilbert-module} can be realized as $T^\ast_{f_1}U T_{f_2}$.
\end{proposition}
\begin{proof}
	Let $\varphi,\psi\in C^\infty_c(G_M, r^\ast W)$.
	Then
	\begin{align*}
	\langle T_{f_1}^\ast UT_{f_2}\varphi, \psi \rangle(\gamma) &= \langle U(f_2\otimes \varphi), f_1\otimes \psi \rangle(\gamma)\\
	&=\int \left\langle U(f_2\otimes \varphi)(\gamma_2,\gamma_1\circ \gamma^{-1}), f_1\otimes \psi(\gamma_2,\gamma_1)\right\rangle  d\mu(\gamma_1)d\mu(\gamma_2)\\
	&=\int \left\langle f_2(\gamma_2)\cdot \gamma_2.\varphi(\gamma_2^{-1}\circ \gamma_1\circ \gamma^{-1}), f_1(\gamma_2)\psi(\gamma_1)\right\rangle d\mu(\gamma_1)d\mu(\gamma_2)\\
	&=\int \left\langle f(\gamma_2) \gamma_2.\varphi(\gamma_2^{-1}\circ \gamma_1\circ \gamma^{-1}), \psi(\gamma_1)\right\rangle d\mu(\gamma_1)d\mu(\gamma_2) \\
	&=\langle f.\varphi, \psi\rangle(\gamma),
	\end{align*}
here from the first line to the second line we use \eqref{eq-calculation-formula}. This completes the proof.
\end{proof}

Let $E$ be a finitely generated projective Hilbert $B$-module, $E^\ast$ be the space of adjointable operators between $E$ and $B$. It has naturally a left $B$-action which is given by
$
(b.\varphi)(e) = b.\varphi(e)
$ for $b\in B, e\in E$ and $\varphi\in E^\ast$.

\begin{lemma}
	$E^\ast$ can be given a Hilbert $B$-module structure and
	$E^\ast$ and $E$ are isomorphic. The isomorphism $E\to E^\ast$ is given by sending $e\in E$ to the adjointable operator 
	\[
	E\ni e^\prime\mapsto \langle e, e^\prime\rangle_E \in B.
	\]
	Moreover, $E\otimes_B E^\ast \cong \mathcal{K}_B (E)$.
\end{lemma}

\begin{proof}
	If $E=B^n$ for some integer $n$, an adjointable map $E\to B$ is determined by the images of $(1,0,\cdots,0),(0,1,0,\cdots,0),\cdots,(0,0\cdots,0,1)$ in $B$ which we shall denote by $b_1,b_2\cdots, b_n$. In this case, $E^\ast =B^n$, and the isomorphism is given by sending $(b_1,b_2,\cdots, b_n)\in E$ to $v \mapsto \langle v, (b_1, b_2, \cdots, b_n) \rangle$. In general, $E$ is finitely generated and projective, there is an orthogonal complemented Hilbert $B$-module bundle $E^\perp$ with $E\oplus E^\perp=B^n$. An adjointable operator $E\to B$ can be complemented to an adjointable operator $B^n \to B$ and is given by taking the inner product with some element $w\in B^n$. Let $p$ be the projection from $B^n$ to $E$, then the restriction of the adjointable operator $B^n\to B$ to $E$ is given by sending $v\in E$ to $\langle v, pw\rangle$. Therefore, there is an isomorphism $E\cong E^\ast$. The space $E^\ast$ is a left $B$-module and right Hilbert $\mathcal{K}_B(E)$-module, and the inner tensor product $E\otimes_B E^\ast\cong \mathcal{K}_B(E)$.
\end{proof}

\begin{lemma}\label{lem-unital-hilbert-module-compact-adjointable}
	If $E_1$ and $E_2$ are two finitely generated projective Hilbert modules over some unital $C^\ast$-algebra, then the set of compact operators between $E_1$ and $E_2$ equals the set of adjointable operators between $E_1$ and $E_2$.
\end{lemma}

\begin{proof}
	Let $E_1, E_2$ be finitely generated projective Hilbert modules over unital $C^\ast$-algebra $B$. Then there are complemented Hilbert modules $E^\perp_1, E_2^\perp$ with $E_i\oplus E^\perp_i=B^n$ for some $n\in \mathbb{N}$ and $i=1,2$. Then
	\begin{equation}\label{eq-matrix-of-compact-operators}
	\mathcal{K}_B(E_1\oplus E_1^\perp, E_2\oplus E^\perp_2)=
	\begin{bmatrix}
		\mathcal{K}_B(E_1,E_2) & \mathcal{K}_B(E_1,E^\perp_2)\\
		\mathcal{K}_B(E_1^\perp, E_2) & \mathcal{K}_B(E_1^\perp,E_2^\perp)
	\end{bmatrix}.
	\end{equation}
	On the other hand, since $B$ is unital, $\mathcal{K}_B(E_1\oplus E_1^\perp, E_2\oplus E^\perp_2)=	\mathcal{L}_B(E_1\oplus E_1^\perp, E_2\oplus E^\perp_2)$ and
	\begin{equation}\label{eq-matrix-of-adjointable-operators}
		\mathcal{L}_B(E_1\oplus E_1^\perp, E_2\oplus E^\perp_2)=
		\begin{bmatrix}
			\mathcal{L}_B(E_1,E_2) & \mathcal{L}_B(E_1,E^\perp_2)\\
			\mathcal{L}_B(E_1^\perp, E_2) & \mathcal{L}_B(E_1^\perp,E_2^\perp)
		\end{bmatrix}.
	\end{equation}
By comparing \eqref{eq-matrix-of-compact-operators} with \eqref{eq-matrix-of-adjointable-operators}, we have  $\mathcal{K}_B(E_1,E_2)=\mathcal{L}_B(E_1,E_2)$.
\end{proof}

\begin{proposition}\label{prop-action-of-monodromy-on-hilbert-module}
	The action \eqref{eq-action-groupoid-on-hilbert-module} extends to a $\ast$-homomorphism $C^\ast G_M \to \mathcal{L}(\mathcal{E}_W)$ whose image is contained in the algebra of compact operators on $\mathcal{E}_W$.
\end{proposition}

\begin{proof}
	According to the above discussion, 
	\begin{equation}\label{eq-ineq-key-estimate}
		||f.\varphi|| \leq ||T_{f_1}^\ast||\cdot ||T_{f_2}|| \cdot ||\varphi||,
	\end{equation}
where we omit the norm of $U$ since it is an isometry. It is easy to check that 
\[
||T_f|| = ||f||_{L^2(G_M, s, \mu)} = \sup_{x\in G^{(0)}_M} \left|\int_{s(\gamma)=x} |f(\gamma)|^2 d\mu(\gamma)\right|^{1/2}
\] 
and 
\[
||T^\ast_f|| = ||f||_{L^2(G_M, r, \widetilde{\mu})}=\sup_{x\in G^{(0)}_M} \left|\int_{r(\gamma)=x} |f(\gamma)|^2 d\widetilde\mu(\gamma)\right|^{1/2}.
\]
Let $f_1(\gamma)=|f(\gamma)|^{1/2}$, $f_2(\gamma)=f(\gamma)/f_1(\gamma)$ if $f(\gamma)\neq 0$ and $f_2(\gamma)=0$ if $f(\gamma)=0$. In this way, we have $f=f_1\cdot f_2$ and $|f_1|^2=|f_2|^2=|f|$.
According to the definition $||f||_I = \max\{||T_{f_2}||^2, ||T^\ast_{f_1}||^2\}$. Therefore, the inequality \eqref{eq-ineq-key-estimate} becomes
\begin{equation*}
	||f.\varphi|| \leq ||f||_I\cdot ||\varphi||
\end{equation*}
which completes the extension part of the proof.

Since $B$ is unital, according to Lemma~\ref{lem-unital-hilbert-module-compact-adjointable} the parallel translation along a curve $\gamma\in G_M$ is an element of the space of compact operators $\mathcal{K}(W_{s(\gamma)}, W_{r(\gamma)})$. Therefore the action of $C^\infty_c(G_M)$ is given by convolution multiplication with an element in $C^\infty_c(G_M, r^\ast W \otimes_B s^\ast W^\ast)$. 

It suffices to show that the operator given by convolution multiplication with element in $C^\infty_c(G_M, r^\ast W \otimes_B s^\ast W^\ast)$ is a compact operator. Let $\varphi_0, \psi_0$ be sections of $W\to M$, $f\in C^\infty_c(G_M,B)$, and denote by $\psi^\ast_0$ the section of $W^\ast \to M$ which is given by $\psi^\ast_0(\varphi_0)=\langle \psi_0, \varphi_0\rangle_W$. Since $W$ is a finitely generated projective Hilbert module bundle over $M$, $C^\infty_c(M, W)$ is a finitely generated projective module over $C^\infty_c(M,B)$. Hence, $C^\infty_c(G_M, r^\ast W)$, $C^\infty_c(G_M, s^\ast W^\ast)$ are finitely generated projective modules over $C^\infty_c(G_M,B)$. More precisely, let $\{\varphi_i\}$ be a finite sequence of smooth sections of $W\to M$
such that the span of $\{\varphi_i(m)\}$ is $W_m$ for all $m\in M$. Then $C^\infty_c(G_M, r^\ast W)$ can be obtained as span of $f_i\cdot r^\ast \varphi_i$ where $f_i\in C^\infty_c(G_M,B)$. Similar result holds for $C^\infty_c(G_M, s^\ast W^\ast)$.

Accordingly, $\{\varphi_i(m)\otimes \varphi_j^\ast(n)\}$ span the vector space $W_m\otimes W^\ast_n$ for all $m\in M$ and $n\in M$. Therefore elements in $C^\infty_c(G_M, r^\ast W \otimes_B s^\ast W^\ast)$ can be written as span of 
	\begin{equation}\label{eq-basic-tensor}
		r^\ast\varphi\otimes f\otimes s^\ast\psi^\ast,
	\end{equation}
where $\varphi, \psi\in C^\infty_c(M,W)$, $\psi^\ast\in C^\infty_c(M, W^\ast)$ and $f\in C^\infty_c(G_M, B)$. It suffices to show the operator $T_{\varphi,f,\psi}$ which is given by convolution multiplication with elements of the form \eqref{eq-basic-tensor} is a compact operator.
	
If there are $f_1,f_2\in C^\infty_c(G_M, B)$ such that $f_1\ast f_2 = f$, we shall pick $\varphi_1\in C^\infty_c(G_M, r^\ast W)$ which is given by $\varphi_1(\gamma)= \varphi(r(\gamma))f_1(\gamma)$ and $\psi_1\in C^\infty_c(G_M, r^\ast W)$ which is given by $\psi_1(\gamma^{-1}) = \psi(s(\gamma))f_2(\gamma)$. Then 
\begin{align*}
\theta_{\varphi_1,\psi_1}h (\gamma)&= \varphi_1.\langle \psi_1, h\rangle (\gamma)\\
&= \int_{\gamma_1\in G_{M,s(\gamma)}} \varphi_1(\gamma\circ \gamma_1^{-1})\langle \psi_1, h\rangle(\gamma_1) d\mu_{s(\gamma)}(\gamma_1)\\
&= \int_{\gamma_1\in G_{M,s(\gamma)}} \varphi_1(\gamma\circ \gamma_1^{-1})  d\mu_{s(\gamma)}(\gamma_1) \int_{\gamma_2\in G_{M,s(\gamma_1)}}\left\langle \psi_1(\gamma_2\circ \gamma_1^{-1}), h(\gamma_2)\right\rangle d\mu_{s(\gamma_1)}(\gamma_2) \\
&= \int \varphi(r(\gamma))f_1(\gamma\circ \gamma_1^{-1}) \left\langle \psi(r(\gamma_2))f_2(\gamma_1\circ \gamma_2^{-1}), h(\gamma_2) \right\rangle d\mu_{s(\gamma)}(\gamma_1)d\mu_{s(\gamma_1)}(\gamma_2)\\
&= \int \varphi(r(\gamma))f(\gamma\circ \gamma_2^{-1})\psi^\ast(r(\gamma_2)) h(\gamma_2) d\mu_{s(\gamma)}(\gamma_2)\\
&= \int_{\gamma_2\in G_{M, s(\gamma)}} \left( r^\ast\varphi\otimes f\otimes s^\ast\psi^\ast \right)(\gamma\circ \gamma_2^{-1}) h(\gamma_2) d\mu_{s(\gamma)}(\gamma_2),
\end{align*}
where from the first line to the second line we use the equation
\eqref{eq-action-compactly-pre-hilbert}, from the second line to the third line we use the equation \eqref{eq-action-of-function-on-section}, from the third line to the fourth line we plug-in the definition of $\varphi_1$ and $\psi_1$ and from the fourth line to the fifth line we use the assumption that $f=f_1\ast f_2$. Therefore, the operator $T_{\varphi, f, \psi}$ which is the convolution with the element of the form \eqref{eq-basic-tensor} is a compact operator.

In general, since $C^\ast$-algebras have approximate identity and $C^\infty_c(G_M, B)$ is dense in $C^\ast(G_M, B)$, any $f\in C^\infty_c(G_M,B)$ can be approximated by elements of the form $f_1\ast f_2$ in norm $||\cdot||_{C^\ast (G_M,B)}$. It can be checked that the operator norm of $r^\ast \varphi\otimes f\otimes s^\ast\varphi^\ast$ is less than or equal to \[
||\varphi||_\infty\cdot ||f||_{C^\ast(G_M,B)}\cdot ||\psi||_\infty,
\]
where $||\varphi||_\infty=\sup_{m\in M} ||\varphi(m)||_{W}$ and $||\psi||=\sup_{m\in M} ||\psi(m)||_W$. Then for any $\varepsilon> 0$ there is $f_1,f_2\in C^\infty_c(G_M, B)$ such that $||f-f_1\ast f_2||_{C^\ast(G_M, B)}<\varepsilon/\left(||\varphi||_\infty\cdot ||\psi||_\infty\right)$, and there is a compact operator $\theta_\varepsilon$ such that $||\theta_\varepsilon-T_{\varphi, f,\psi}||\leq \varepsilon$. 
This completes the proof.
	\end{proof}

\begin{corollary}\label{coro-key-kk-theory}
	The Hilbert module $\mathcal{E}_W$ together with the zero operator $\left( \mathcal{E}_W, 0\right)$ form a Kasparov module which defines an element in $KK(C^\ast G_M, C^\ast (G_M, B))$.
	\qed
\end{corollary}

\section{Enlargeable foliation}\label{sec-enlargeable-foliation}

\begin{definition}\label{def-enlargeable-foliation}
	A foliation $(M,F)$ is compactly enlargeable if there is $C>0$ such that for any $\varepsilon> 0 $, there is a compact covering $\widetilde{M}$ of $M$ and a smooth map
	\[
	f: \widetilde{M} \to S^n
	\]
	with
	\begin{itemize}
		\item 
		$|f_\ast X| \leq \varepsilon |X|$ for all $X\in C^\infty(\widetilde{M},\widetilde{F})$, where $\widetilde{F}$ is the lifting of $F$ to $\widetilde{M}$;
		\item
		$|f_\ast X| \leq C\cdot |X|$ for all $X\in C^\infty(\widetilde{M}, T\widetilde{M})$;
		\item
		$f$ has nonzero degree.
	\end{itemize}
\end{definition}

\begin{remark} 
	Notice that our notion of enlargeable foliation is in between that of \cite{Zhang20} and \cite{BenameurHeitsch19}, where \cite{Zhang20} only requires enlargeability in the leaf direction and \cite{BenameurHeitsch19} requires enlargeability in all directions.
\end{remark}

Pick a complex vector bundle $E_0$ over sphere $S^n$ such that all its Chern classes vanish except the top-degree one $c_n(E_0)\neq 0$. Let $\widetilde{M}_i$ be the compact cover with covering group $G$ and the constant $\varepsilon=1/i$, the pull back bundle $f^\ast E_0$ can be extended to a $G$-equivariant bundle 
\[
\bigoplus_{g\in G} g^\ast (f^\ast E_0) \to \widetilde{M}_i
\]
which can be reduced to a vector bundle $E_i$ over $M$. As a result all Chern classes of $E_i$ vanish except the top degree one $c_n(E_i)\neq 0$. We shall denote by $P_i \to M$ the frame bundle of $E_i$ which are by themself principal $U(d_i)$ bundles. They are equipped with natural connections whose leafwise curvatures tend to zero as $i\to \infty$.

In the following discussion, we shall make use of several $C^\ast$-algebras $A, A^\prime, Q$ and their variations. The definitions are given in Definition~\ref{def-algebra-A-Q} and Definition~\ref{def-algebras-qaq}.
Let $q_i$ denote the image of $1\in U(d_i)$ on $\mathcal{K}$. We shall consider the family of Hilbert $q_i\mathcal{K}q_i\cong M_{d_i}(\mathbb{C})$-module bundles
\begin{equation}\label{eq-def-of-V-i}
V_i=P_i \times_{U(d_i)} q_i\mathcal{K}q_i,
\end{equation}
where $U(d_i)$ acts on $\mathcal{K}$ by matrix multiplications. 
We shall briefly explain how they can be assembled into a leafwise flat Hilbert $qQq$-module (see \cite[Sec~2]{HankeSchick06} for a detailed construction). Indeed, let $\{U_\alpha\}$ be an open cover of $M$ over which each $V_i$ is trivializable and each $U_\alpha$ is homeomorphic to an unit open disc $(0,1)^n$. We can choose local trivializations 
\begin{equation}\label{eq-local-trivializations}
	\psi_{\alpha,i}: V_i|_{U_\alpha}\to U_\alpha \times q_i \mathcal{K}q_i
\end{equation}
as in \cite[Sec~2]{HankeSchick06} such that 
\begin{equation}\label{eq-choice-of-local-trivialization}
	\nabla^i_{\frac{\partial}{\partial x_k}} s = 0
\end{equation}
if $s$ is a smooth section which is constant, under the trivialization, in $[0,1]^k \times \{0\} \times \cdots \times\{0\}$ the first $k$ variable of $U_\alpha$. Here $\nabla^i$ is the connection on $V_i$.

The corresponding transition functions is denoted by
\begin{equation*}
	\varphi_{\alpha,\beta, i}: U_{\alpha}\cap U_{\beta} \to \End(q_i\mathcal{K}q_i)\cong q_i\mathcal{K}q_i.
\end{equation*}
Since the norm of curvature of $V_i$ is universally bounded with respect to $i\in \mathbb{N}$. According to \cite[Lem~2.3, Lem~2.5 and Prop~2.6]{HankeSchick06}, $\varphi_{\alpha,\beta,i}$ is a Lipschitz function with Lipschitz constant independent of $i$. Therefore, the transition functions can be assembled into 
\begin{equation}\label{eq-assembling-transition-functions}
	\varphi_{\alpha, \beta}=\left(\varphi_{\alpha, \beta,1},\varphi_{\alpha, \beta,2},\cdots,\varphi_{\alpha, \beta,i},\cdots\right),
\end{equation}
which is a Lipschitz map from $U_\alpha\cap U_\beta$ to $qAq$. It determines a Lipschitz Hilbert $qAq$-module bundle over $M$ which can be approximated by a smooth Hilbert $qAq$-module bundle $V$ over $M$. 

The properties of bundle $V$ are summarized in the following.

\begin{proposition}
	There is a Hilbert $qAq$-module bundle $V$ over $M$ such that 
	\begin{itemize}
		\item 
		$V_i$, defined in \eqref{eq-def-of-V-i}, is isomorphic to $V\cdot q_iA_iq_i$ as Hilbert $\mathcal{K}$-module bundle;
		\item
		The connection of $V$ preserves subbundle $V_i$;
		\item
		The leafwise curvature take values in $\hom(qAq,qA^\prime q)$.
	\end{itemize}
\end{proposition} 
Therefore the bundle $W=V/V\cdot qA^\prime q$ is a leafwise flat Hilbert $qQq$-module bundle which, according to Corollary~\ref{coro-key-kk-theory}, determines an element in $KK(C^\ast G_M, C^\ast (G_M, qQq))$.  The $KK$-element induces $(\phi_1)_\ast: K_0(C^\ast G_M) \to K_0(C^\ast (G_M, qQq))$. The above procedure can be replicated if we start with a sequence of trivial principal bundles $\{P^\prime_i\}$ with $P^\prime_i = M\times M_{d_i}(\mathbb{C})$. We shall get a new $KK$-theory element in $KK(C^\ast G_M, C^\ast (G_M, qQq))$ and corresponding $(\phi_2)_\ast: K_0(C^\ast G_M) \to K_0(C^\ast (G_M, qQq))$. Let
\begin{equation}\label{eq-def-phi}
\phi_\ast=(\phi_1)_\ast- (\phi_2)_\ast.
\end{equation}

Recall that the Rosenberg index $[\alpha]\in K_0(C^\ast G_M)$ is given in Definition~\ref{def-rosenberg-index}.
\begin{proposition}\label{prop-pushforward-of-dirac}
	Let $[D_W]\in K_0(C^\ast (G_M, qQq))$ denote the image of $[\alpha]\in K_0(C^\ast G_M)$ under the map $(\phi_1)_\ast: K_0(C^\ast G_M) \to K_0(C^\ast (G_M, qQq))$. 
	Then $[D_W]$ coincides with the Rosenberg index twisted by the leafwise flat Hilbert $qQq$-module bundle $W$.
\end{proposition}

\begin{proof}
	$[\alpha]$ is given by the Kasparov module $\left(\mathcal{E}, f(D)\right)$ while the $KK$-theory element is given by the Kasparov module $(\mathcal{E}_W,0)$. Their Kasparov product is given by the pair $\left( \mathcal{E}\otimes_{C^\ast G_M} \mathcal{E}_W, f(D)\otimes 1\right)$. According to the definition, the inner tensor product is completion of $\mathcal{E}\otimes_{\operatorname{alg}} \mathcal{E}_W/N$ where $N$ is the span of elements of the form 
	$$\varphi.a\otimes \psi-\varphi\otimes \Theta(a)\psi$$ 
	with $\varphi\in \mathcal{E}, a\in C^\ast G_M, \psi\in \mathcal{E}_W$ and $\Theta: C^\ast G_M \to \mathcal{L}(\mathcal{E}_W)$ being the map defined in Proposition~\ref{prop-action-of-monodromy-on-hilbert-module}. Consider the following map $\pi: C^\infty_c(G_M, r^\ast S)\otimes_{\operatorname{alg}} C^\infty_c(G_M, r^\ast W) \to C^\infty_c(G_M, r^\ast S \otimes r^\ast W)$ given by
	\begin{equation}\label{eq-iso-of-hilert-module-tensor}
	\pi(\varphi\otimes \psi)(\gamma)=  \int_{G_{M,s(\gamma)}} \varphi(\gamma\circ \gamma_1^{-1})\otimes (\gamma\circ \gamma_1^{-1}). \psi(\gamma_1) d\mu(\gamma_1),
	\end{equation}
	where $\varphi\in C^\infty_c(G_M, r^\ast S), \psi\in C^\infty_c(G_M, r^\ast W)$ and $(\gamma\circ \gamma_1^{-1}). \psi(\gamma_1)$ is the parallel translation of $\psi(\gamma_1)$ along the curve $\gamma\circ \gamma_1^{-1}$. It is a matter of direct calculation to check that $\pi$ vanishes on $C^\infty_c(G_M, r^\ast S)\otimes_{\operatorname{alg}} C^\infty_c(G_M, r^\ast W)\cap N$ and preserves the inner product if taken as map from $\mathcal{E}\otimes_{C^\ast G_M} \mathcal{E}_W$ to  the  completion of $C^\infty_c(G_M, r^\ast S \otimes r^\ast W)$.
	
	The covariant derivative on $S\otimes W$ is given by $\nabla^{S\otimes W} = \nabla^S\otimes 1+ 1\otimes \nabla^W$. We have
	\[
	\nabla^{S\otimes W}_{e_i} \pi(\varphi\otimes \psi) (\gamma) = \int_{G_{M,s(\gamma)}} \nabla_{e_i}^S\varphi(\gamma\circ \gamma_1^{-1})\otimes (\gamma\circ \gamma_1^{-1}). \psi(\gamma_1) d\mu(\gamma_1),
	\]
	where $1\otimes \nabla^W$ does not appear because $(\gamma\circ \gamma_1^{-1}).\psi(\gamma_1)$ is, by definition, parallel with respect to the curve $\gamma$ and the connection $\nabla^W$. So the operator $f(D)\otimes 1$ is precisely $f(D_W)$ under the identification \eqref{eq-iso-of-hilert-module-tensor}.
\end{proof}

	By the same reason, the image of $[\alpha]$ under the map $(\phi_2)_\ast$ is the Rosenberg index $[D_{qQq}]$ twisted by the trivial bundle $M\times qQq$. Let $\pi:  A\to Q$ be the canonical projection, it induces $\pi_\ast: K_0(C^\ast (G_M, qAq)) \to K_0(C^\ast (G_M , qQq))$. Let $[D_V], [D_{qAq}]\in K_0(C^\ast (G_M, qAq))$ be the elements defined by the leafwise Dirac-type operators twisted by the non-flat bundle $V$ and the trivial bundle $M\times qAq$ respectively. Then it is straightforward to verify that we have $\pi_\ast [D_V] = [D_W]$ and $\pi_\ast [D_{qAq}] = [D_{qQq}]$. (see also \cite[Lem~3.1]{HankeSchick06}).

Consider the following composition: 
\begin{equation}\label{eq-detecting-non-vanishing}
	K_0(C^\ast (G_M, qAq)) \to 	K_0(C^\ast G_M)   \to K_0(C^\ast_r G_H),
\end{equation}
where the first arrow is given by the homomorphism sending $A$ to its $i$-th component $\mathcal{K}$, and the second arrow is given by \eqref{eq-hom-from-monodromy-to-reduced-holonomy}.  

\begin{proposition}\label{prop-index-pushforward}
	The image of $[D_V]$ under the map \eqref{eq-detecting-non-vanishing}
	is computed by the longitudinal index element corresponds to $D_{E_i}$.
\end{proposition}

\begin{proof}
	It is a consequence of Remark~\ref{remark-relation-between-rosenberg-index-and-longitudinal-index}.

\end{proof}



\begin{proposition}\label{eq-compute-k-theory-of-new-algebra}
	$K_0(C^\ast (G_M, qA^\prime q))=\bigoplus K_0(C^\ast G_M)$.
\end{proposition}

\begin{proof}
	By the Dini theorem, the subspace 
	$$
	\bigoplus C^\infty_c(G_M, q_i\mathcal{K}q_i)\subset C^\infty_c(G_M, qA^\prime q)
	$$ 
	is dense in the I-norm. It is clear that $\bigoplus_{i=1}^k q_i\mathcal{K}q_i$ is an ideal in $qA^\prime q$ for all $k\in \mathbb{N}$. According to Proposition~\ref{prop-exact-sequence-of-new-C-algebra}, we have the inclusion $\bigoplus_{i=1}^k C^\ast (G_M, q_i\mathcal{K}q_i) \subset C^\ast (G_M, qA^\prime q)$ for all $k\in \mathbb{N}$. Therefore, the $C^\ast$-algebra $C^\ast(G_M, qA^\prime q)$ can be realized as direct limit of $\bigoplus C^\ast (G_M, q_i \mathcal{K} q_i)$.
\end{proof}


\begin{proposition}
	Let $\phi_\ast: K_0(C^\ast G_M) \to K_0(C^\ast (G_M, qQq))$ be defined as in \eqref{eq-def-phi}. Then $\phi_\ast [\alpha] \neq 0$ in $K_0(C^\ast (G_M, qQq))$.
\end{proposition}

\begin{proof}
	By Proposition~\ref{prop-pushforward-of-dirac}, the image of $[D_V]-[D_{qAq}] \in K_0(C^\ast (G_M, qAq))$ under the map 
	\[
	\pi_\ast: K_0(C^\ast (G_M , qAq)) \to K_0(C^\ast (G_M, qQq))
	\]
	is precisely $\phi_\ast [\alpha]$. By exact sequence \eqref{eq-exact-sequence-algebra}, it suffices to show that $[D_V]-[D_{qAq}]\in K_0(C^\ast (G_M, qAq))$ does not come from the image of $K_0(C^\ast (G_M, qA^\prime q))$.
	
	Consider the following commutative diagram
	\[
	\xymatrix{
		K_0(C^\ast (G_M, qA^\prime q)) \ar[r] \ar[dr]& K_0(C^\ast (G_M, qAq)) \ar[d] \\
		& \prod K_0(C^\ast G_M),
	}
	\]
	where the downward arrows are given by sending $A$ and $A^\prime$ to $A_i$'s. It then suffices to show the image of $[D_V]-[D_{qAq}]$ under the vertical downward arrow has infinitely many nonzero terms. 
	
	Indeed, according to Proposition~\ref{prop-index-pushforward}, under the map $K_0(C^\ast G_M) \to K_0(C^\ast_r G_H)$ the $i$-th component of $[D_V]-[D_{qAq}]$ is given by the longitudinal index of the Dirac type operator twisted by the virtual bundle $E_i-\mathbb{C}^{d_i}$. According to Connes\cite{Connes83}, there is a transverse fundamental class $\mu$ such that 
	\begin{equation*}
		\mu([D_{E_i-\mathbb{C}^{d_i}}]) = \langle \widehat{A}(F)\operatorname{ch}(E_i-\mathbb{C}^{d_i}), [M]\rangle,
	\end{equation*}
	where $[M]$ is a fundamental class of $M$. According to our non-vanishing assumption of top Chern classes, the sequence $\mu([D_{E_i-\mathbb{C}^{d_i}}]) $ is nonzero for all $i$. This contradicts with Proposition~\ref{eq-compute-k-theory-of-new-algebra}.

\end{proof}

The above proposition directly implies our main theorem.

\begin{theorem}\label{thm-main}
	If $(M,F)$ is a compactly enlargeable foliation in the sense of Definition~\ref{def-enlargeable-foliation} with $F$ spin and even dimensional, then $[\alpha]\neq 0$ in $K_0(C^\ast G_M)$.
\end{theorem}

\section{Reduction to the even dimensional case}\label{sec-reduction}

If the foliation $F\to M$ is of odd dimensional, the exterior product of vector bundles $F\boxtimes TS^1\to M\times S^1$ defines an even dimensional foliation. Moreover, the monodromy groupoid of $(M\times S^1, F\boxtimes TS^1)$ is the direct product of monodromy groupoid of $(M,F)$ and the fundamental groupoid of $S^1$. Accordingly, the corresponding maximal groupoid $C^\ast$-algebra is $C^\ast G_M\otimes \mathcal{K}\otimes  C^\ast \mathbb{Z}$ whose $K$-theory is computed by the universal coefficient theorem:
\begin{equation*}
	K_\ast(C^\ast G_M\otimes \mathcal{K}\otimes C^\ast \mathbb{Z}) = K_\ast(C^\ast G_M)\otimes K_\ast (C^\ast \mathbb{Z}).
\end{equation*}
Here $\mathcal{K}$ is the $C^\ast$-algebra of compact operators and $C^\ast \mathbb{Z}$ is the group $C^\ast$ algebra of $\mathbb{Z}$. In particular, we have 
\begin{equation*}
	K_0(C^\ast G_M \otimes \mathcal{K} \otimes C^\ast \mathbb{Z}) = K_0(C^\ast G_M)\otimes 1 \oplus K_1(C^\ast G_M) \otimes e,
\end{equation*}
where $1$ is the generator of $K_0(C^\ast \mathbb{Z})$ and $e$ is the generator of $K_1(C^\ast \mathbb{Z})$. As in \cite{HankeSchick06}, $[\alpha(M,F)]\in K_1(C^\ast G_M)$ is defined by requiring 
\begin{equation*}
	[\alpha(M,F)]\otimes e = [\alpha(M\times S^1, F\boxtimes TS^1)] \in K_0(C^\ast G_M \otimes \mathcal{K} \otimes C^\ast \mathbb{Z}).
\end{equation*}

\begin{proposition}\label{prop-reduction}
	The foliation $(M\times S^1, F\boxtimes TS^1)$ is compactly enlargeable if $(M,F)$ is compactly enlargeable.
\end{proposition}

\begin{proof}
	Assume that $(M,F)$ is compactly enlargeable. Then for any $\varepsilon > 0$ there is compact covering space $\widetilde{M_{\varepsilon}} \to M$ and map $f_\varepsilon:  \widetilde{M_{\varepsilon}}  \to S^n$ with the properties given in Definition~\ref{def-enlargeable-foliation}. Since $S^1$ is also enlargeable, there is $g_\varepsilon: S^1_\varepsilon\to S^1$ with the properties of Definition~\ref{def-enlargeable-foliation}. Fix a degree one map $\varphi: S^n\times S^1\to S^{n+1}$ and let $C_1=\max |\varphi_\ast|$, we claim that the following composition 
	\begin{equation*}
		 \widetilde{M_\varepsilon} \times S^1_\varepsilon \xrightarrow{(f_\varepsilon, g_\varepsilon)} S^n\times S^1\xrightarrow{\varphi} S^{n+1}
	\end{equation*}
has the wanted property. Indeed, let $\widetilde{F}$ is the lifting of $F$ to $\widetilde{M_{\varepsilon}}$ and $\Phi=\varphi\circ (f_\varepsilon, g_\varepsilon)$  then
\[
|\Phi_\ast (X,Y)|=|\varphi_\ast (f_{\varepsilon,\ast}X, g_{\varepsilon,\ast}Y)|\leq \varepsilon C_1 |(X,Y)|,
\]
for $(X,Y)\in C^\infty(\widetilde{M_\varepsilon}\times S^1_\varepsilon, \widetilde{F}\boxtimes TS^1_\varepsilon)$ and 
\[
|\Phi_\ast(Z,Y)| \leq CC_1 |(Z,Y)|,
\]
for  any tangent vector $(Z,Y)$ of $\widetilde{M_\varepsilon}\times S^1_\varepsilon$.
This completes the proof.
\end{proof}

According to Theorem~\ref{thm-main} and Proposition~\ref{prop-reduction}, $[\alpha(M\times S^1, F\boxtimes TS^1)]$ is nonzero, so $[\alpha(M,F)]$ is also nonzero.

\section*{Acknowledgements}
	The authors would like to thank Professor Weiping Zhang for kindly suggesting this problem. The authors want to thank Professor Georges Skandalis for very helpful guidance on groupoid $C^*$-algebras, in particular for pointing out Proposition~\ref{prop-canonical-map-monodromy-to-holonomy}. The authors want to thank Professor Alexander Engel for pointing out a mistake in a previous version.
	

\bibliography{Refs} 
\bibliographystyle{amsalpha}

\noindent {\small   Chern Institute of Mathematics and LPMC, Nankai University, Tianjin 300071, P. R. China. }

\smallskip

\noindent{\small Email: guangxiangsu@nankai.edu.cn}

\medskip

\noindent{\small School of Mathematical Sciences, Tongji University, Shanghai 200092, P. R. China.}

\smallskip

\noindent{\small Email: zelin@tongji.edu.cn}

\end{document}